\documentclass[english,11pt]{smfart}
 \pdfoutput=1

\usepackage{amsmath,amssymb,smfthm}
\usepackage{mathrsfs}
\usepackage{pinlabel}
\usepackage{hyperref}

\newcommand{\Hyp}{\mathbb{H}}
\newcommand{\HH}{\mathbb{H}}
\newcommand{\HHb}{\overline{\mathbb{H}^2}}
\newcommand{\RR}{\mathbb R}
\newcommand{\ZZ}{\mathbb Z}
\newcommand{\RRb}{\overline{\mathbb R}}
\newcommand{\HxR}{\Hyp^2\times\mathbb{R}}
\newcommand{\del}{\partial}
\newcommand{\calC}{\mathcal C}

  \newcommand{\HR}{X}
  
  \newcommand{\Gcomp}[1]{\overline{#1}^g}
  \newcommand{\Gbound}{\operatorname{\partial_g}}
  \newcommand{\Pcomp}[1]{\overline{#1}^\times}
  \newcommand{\Pbound}{\operatorname{\partial_\times}}

  \newcommand{\argtanh}{\operatorname{\tanh^{-1}}}
\newcommand{\calU}{\mathcal U}

\newcommand{\calF}{\mathcal F}

\newtheorem{prob}{Problem}

  \title[Minimal surfaces in $\HxR$]{On the asymptotic
        behavior of minimal surfaces in $\HxR$}

  \author{Beno\^{\i}t R. Kloeckner}
  \author{Rafe Mazzeo}

\begin{document}

\begin{abstract}
We consider the asymptotic behaviour of properly embedded minimal surfaces 
in $\HxR$ taking into account the fact that there is more than one natural
compactification of this space. This provides a better setting in which to 
consider the general problem of determining which curves at infinity are the
asymptotic boundary of such minimal surfaces. We also construct some new
examples of such surfaces and describe the boundary regularity.
\end{abstract}

\maketitle

\section{Introduction}

The study of minimal surfaces in $\HxR$ has received considerable attention in the past decade or more, 
beginning with the work of Nelli and Rosenberg \cite{Nelli-Rosenberg,Nelli-Rosenberg_err}; we refer to 
\cite{Mazet-Rodriguez-Rosenberg,Masaltsev,Kim-Koh-Shin-Yang, Espinar-Rodriguez-Rosenberg,Morabito,Pyo,Martin-Mazzeo-Rodriguez} 
as well as other references cited below.  This is part of a larger effort to understand minimal surfaces in each of the three-dimensional 
homogeneous (Thurston) geometries. The understanding of minimal surfaces in $\RR^3$, $S^3$ and $\HH^3$ is fairly advanced, 
but by contrast, although $\HxR$ is the simplest non-constant-curvature geometry of this type, the behavior of minimal surfaces in it 
is much less well understood. 

Our initial motivation in the work leading to this paper was to consider this from
a much more general point of view, namely as a stepping stone toward the study
of complete properly embedded minimal submanifolds in general symmetric spaces of noncompact
type and arbitrary rank. The basic question is the asymptotic Plateau problem, in which one asks 
which curves on the asymptotic  boundary of such a space can be ``filled'' by minimal surfaces. 
However, these spaces admit many useful but non-equivalent compactifications, which adds both ambiguity
and complexity to this problem.  Our goal here is to study some aspects of
this problem in $\HxR$, which is a particularly simple type of rank $2$ symmetric space.
Strictly speaking, this space is reducible and so does not exhibit the true geometric
complexity of a standard rank $2$ space such as $\mathrm{SL}(3)/\mathrm{SO}(3)$;
however, it is hyperbolic in some directions and Euclidean in others, which leads to
some surprising phenomena. Furthermore, it does have two interesting and natural
compactifications: the `product' compactification and the geodesic compactification.
We explain these in some detail below, but briefly, the product compactification
is the one obtained as the product of the compactifications of each of the factors,
$\overline{\HH^2} \times \overline{\RR}$, while the geodesic compactification is
a closed $3$-ball obtained by attaching an endpoint to each infinite geodesic
ray emanating from a given point $p$. 

Nelli and Rosenberg \cite{Nelli-Rosenberg} proved that every simple closed curve 
lying in the vertical boundary $\del \HH^2 \times \RR$ which is a vertical graph 
(over $\del \HH^2 \times \{0\}$) is the asymptotic boundary of a minimal surface -- or, 
as we shall say, is minimally fillable. On the other hand, \cite{SaEarp-Toubiana} shows that many curves 
on this vertical boundary are not minimally fillable.  We may also consider curves which 
reach $\del \HH^2 \times \{\pm \infty\}$, or which cross into the horizontal parts of this 
boundary (i.e., the caps at $\del \RR$). We explain below a general result, Proposition \ref{theo:fillability},
which guarantees fillability of a wide class of curves, but also present curves, in Theorem \ref{theo:butterfly}, 
which are minimally fillable but do not satisfy the hypotheses of Proposition \ref{theo:fillability}. 

We also consider the fillability question for curves in the geodesic boundary; this
seems not to have been considered explicitly before, and was the original starting
point of our work. It turns out that in fact almost no curve in this geodesic
boundary is minimally fillable. Our second main result Theorem \ref{theo:oscillations} 
shows that a minimal surface with a nontrivial portion of its asymptotic boundary 
contained in the geodesic boundary of $\HxR$ must oscillate so greatly that its 
boundary cannot be a curve, and often has nonempty interior. We present some 
examples which exhibit this behaviour in Theorem \ref{theo:examples}.   

The research which led to this paper was undertaken some years ago, but for
various reasons this manuscript was not completed. In the intervening time,
an interesting new paper by Coskunuzer \cite{Coskunuzer} has appeared
which addresses certain of these same questions, but from a somewhat
different point of view. While there are obvious points of intersection 
between these two papers, the present paper contains a number of different
results and perspectives which we hope will be of interest. 




\subsection*{Acknowledgements}
The first author is grateful to Anne Parreau and Fran\c cois Dahmani for interesting discussions related 
to Theorem \ref{theo:examples}. The second author acknowledges many useful conversations
about this general subject with Francisco Martin and Magdalena Rodriguez.

\section{Compactifications}

In this first section we set some notation and describe in 
more detail the two compactifications of $\HR$ mentioned 
above

\subsection{Notation}

To simplify notation, we often write $\HH^2 \times \RR$ as $\HR$.  Let $g$
be the standard product metric on this space, and $d$ the induced distance function.
The distance in the horizontal factor $\Hyp$ is denoted by $d_\Hyp$.

We consider two compactifications of $\HR$: the product compactification $\Pcomp{\HR}$ and the geodesic 
compactification $\Gcomp{\HR}$. Each is obtained by adjoining to $\HR$ an asymptotic boundary, $\Pbound\HR$ and 
$\Gbound\HR$, respectively. Given any properly embedded surface $\Sigma \subset \HR$, we let 
$\Pcomp{\Sigma}$ denote the closure of $\Sigma$ in $\Pcomp{\HR}$ and then write $\Pbound\Sigma:=
\Pcomp\Sigma\cap\Pbound\HR$. We use the corresponding notation for the analogous sets in the geodesic compactification. 

The two factors in $\HR$ have rank one (i.e., the maximal flat totally geodesic subspaces are one-dimensional),
hence all of their classical compactifications are equivalent. We denote these by 
\[
\overline{\Hyp}^2=\Hyp^2\cup \partial \Hyp^2 \quad\mbox{and}\quad 
\overline{\mathbb{R}}=\mathbb{R} \cup \partial \mathbb{R}. 
\] 

\subsection{The product compactification of $\HxR$}

The product compactification of $\HR$ is the product 
\[
\Pcomp{\HR}=\overline{\Hyp}^2\times \overline{\mathbb{R}}
\]
with the product topology (see Figure \ref{fig:product}). Its boundary $\Pbound \HR=\Pcomp{\HR}\setminus\HR$
is the disjoint union of three parts:
\[
\Pbound \HR= \big(\partial \Hyp^2 \times \mathbb{R}\big) \cup 
            \big(\partial \Hyp^2 \times \partial \mathbb{R}\big) \cup
            \big(\Hyp^2 \times \partial \mathbb{R}\big)
\]
The first component, which is an open cylinder, is called the vertical boundary; the third is 
the union of two open disks, which we call the upper and lower caps. The middle component
is the product of the boundaries of each factor and is the union of two circles. 
In particular, $\Pcomp{\HR}$ is a manifold with boundary and corners of codimension $2$. 

\begin{figure}[tb]
\labellist
\small \hair 3pt
\pinlabel $\Hyp^2$ [r] at 0 60
\pinlabel {singular circles} [Bl] at 155 58
\pinlabel cap at 58 87
\pinlabel cap at 60 11
\pinlabel $\partial\Hyp^2\times\mathbb{R}$ [r] at 41 34
\pinlabel $\mathbb{R}$ [l] at 76 69
\endlabellist
\centering
\includegraphics[scale=1.5]{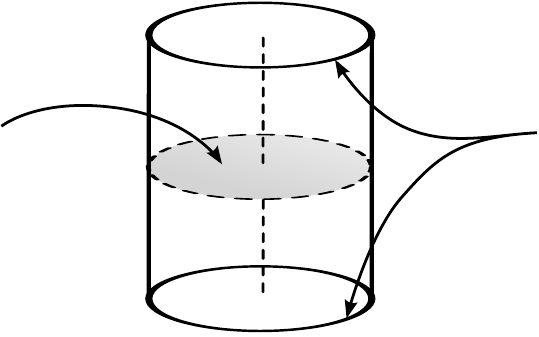}
\caption{The product compactification of $\HxR$.}
\label{fig:product}
\end{figure}

\begin{rema}
This product compactification only makes sense for a product of rank one spaces. For $\HxR$, it is essentially 
isomorphic to more classical compactifications (Satake-Furstenberg, Chabauty, \dots), but this identification 
is not quite true because of issues related to the one-dimensional Euclidean factor. 

To be more specific, the Chabauty compactification is defined as follows: each point in $\HR$ is identified with its 
isotropy group, and the compactification is obtained by taking the closure of the image of this map with respect to 
the so-called Chabauty topology on closed subgroups of the isometry group. This procedure does not distinguish 
between the two caps, $\HH^2 \times \{+\infty\}$ and $\HH^2 \times \{-\infty\}$, so these are glued together 
and the resulting compactification is a solid torus. Other compactifications have different undesirable features, 
again because of this Euclidean factor. This product compactification, although perhaps ad hoc, serves our purposes well.
\end{rema}

\subsection{The geodesic compactification of $\HxR$}

Now recall the geodesic compactification.  This is defined for any simply connected, 
non-positively curved manifold, see \cite{Eberlein}, but we discuss it only for $\HR$. 

\subsubsection{Construction}
We work only with geodesics with constant speed parametrizations, and include the case of speed zero, 
which are constant maps. A geodesic of speed $1$ is called a \emph{unit} geodesic.  A \emph{ray} 
is a geodesic defined on $\RR^+$. The set of all rays in $\HR$ is denoted by $\mathcal{R}$ and the subset 
of unit rays is written $\mathcal{R}_1$.

Two rays $\gamma,\sigma$ are \emph{asymptotic} at $+\infty$, $\gamma\sim \sigma$,  if they remain at bounded distance 
from one from another in the future: 
\[ 
\sup_{t \geq 0} d(\gamma(t),\sigma(t)) < \infty.
\]
This is an equivalence relation, and the class of a ray $\gamma$ is denoted $[\gamma]$. Given $x\in \HR$ 
and $u\in T_x \HR$, then
\[
\gamma(x,u)(t) := \exp_x(tu)
\]
is the unique ray with initial position $x$ and velocity $u$. 

The \emph{geodesic boundary} of $\HR$ is the set of classes of unit rays:
\[
\Gbound\HR:= \mathcal{R}_1 / \sim
\]
endowed with the following topology. For any $x\in\HR$, there is a natural map from the unit sphere in the tangent space 
to this geodesic boundary: 
\begin{eqnarray*}
\pi_x := S_x\HR &\to& \Gbound\HR \\
  u &\mapsto& [\gamma(x,u)]
\end{eqnarray*}
This is a bijection, so we endow $\Gbound\HR$ with the topology of $S_x\HR$ under this identification.
For any two points $x,y$, $\pi_x^{-1}\circ \pi_y$ is a homeomorphism, so this topology does not depend 
on the choice of $x$.

The \emph{geodesic compactification} of $\HR$ is the disjoint union 
\[ 
\Gcomp{\HR} := \HR \cup \Gbound\HR;
\]
with topology determined by the requirements that $\HR$ is open in $\Gcomp{\HR}$ and inherits its usual topology,
and a sequence $x_n$ in $\HR$ converges to a boundary  point if and only if for some (and hence all) 
$y\in \HR$, $d(y,x_n)\to \infty$ and the sequence $u_n$ of unit tangent vectors at $y$ defined by 
$x_n\in\gamma(y,u_n)$ converges to some $u$; then
the limit of $x_n$ is $[\gamma(y,u)]$, which does not depend on $y$.

The topology defined earlier on $\Gbound\HR$ coincides with the induced topology. 
With this topology, $\Gcomp{\HR}$ is homeomorphic to a closed ball, and the isometry 
group of $\HR$ extend to an action by homeomorphisms on the compactification.

\subsubsection{Structure}
Each unit ray $\gamma$ in $\HR$ can be written as $(\gamma_1,\gamma_2)$ where $\gamma_1$ and $\gamma_2$ are
rays in $\Hyp^2$ and $\RR$ with speeds $s_1$ and $s_2$ respectively, where $s_1^2+s_2^2=1$. 
If $\gamma'=(\gamma'_1,\gamma'_2)$ is another unit ray together with projections of speed $s'_i$, then $\gamma\sim \gamma'$ 
if and only if $\gamma_i\sim \gamma'_i$ for $i=1,2$.  This implies that $s_i=s'_i$ (in particular, 
any two constant geodesics are asymptotic to one other).

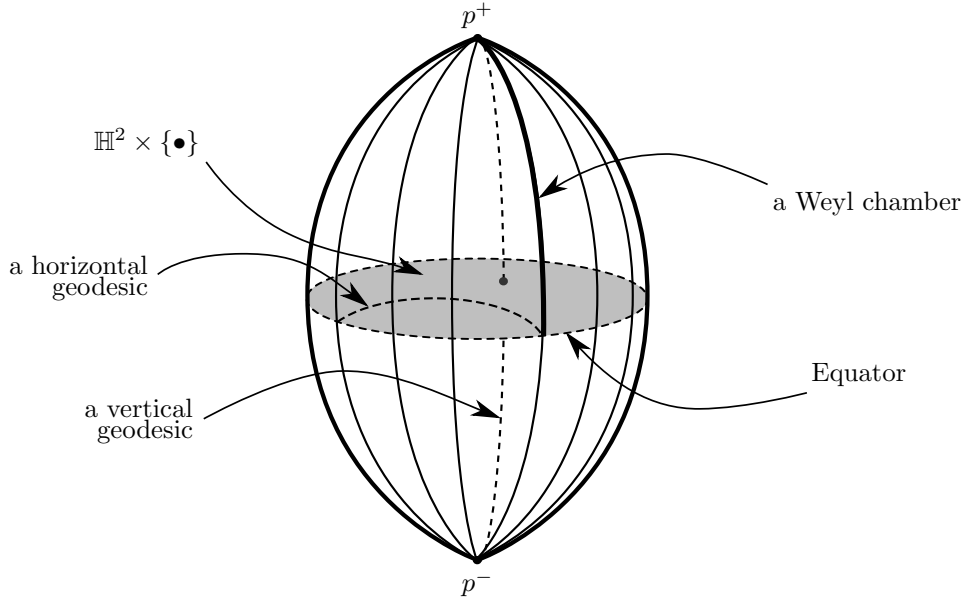
\begin{figure}[tb]
\labellist
\small \hair 3pt
\pinlabel {a Weyl chamber} [tl] at 340 190
\pinlabel {Equator} [bl] at 358 93
\pinlabel $p^+$ [b] at 205 260
\pinlabel $p^-$ [t] at 205 13
\pinlabel $\Hyp^2\times\{\bullet\}$ [br] at 78 201
\pinlabel {\parbox{13ex}{\flushright a horizontal geodesic}}  [r] at 55 148
\pinlabel {\parbox{11ex}{\flushright a vertical geodesic}} [r] at 76 81
\endlabellist
\centering
\includegraphics[scale=0.8]{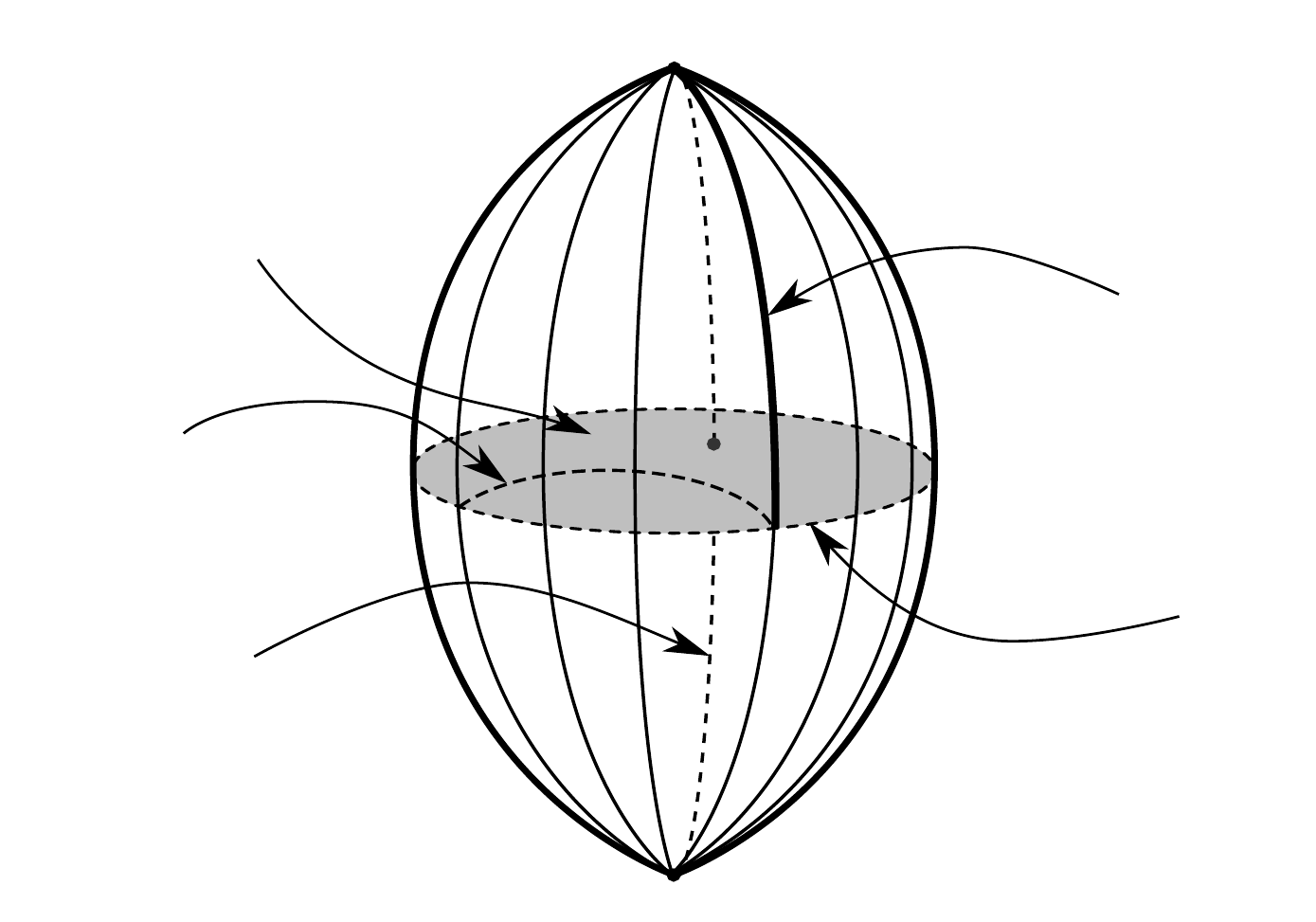}
\caption{The geodesic compactification of $\HxR$. Note that all
horizontal planes $\HH^2\times\{\bullet\}$ are asymptotic, i.e.
have their boundary equal to the same curve, the equator.}
\label{fig:geodesic}
\end{figure}

When $s_1=0$, then $\gamma$ projects to a constant in the $\Hyp^2$ factor and
we  say that $\gamma$ is \emph{vertical}. There are two classes of vertical geodesic, 
those that go up and those that go down, corresponding to the two points of $\del \RR$.
These are called the \emph{poles} and denoted $p^\pm$. 

imilarly, when $s_2=0$, then $\gamma$ is called \emph{horizontal}. Since $(\gamma_1, t_0) \sim
(\gamma_1, t_0')$ for any $t_0, t_0' \in \RR$, asymptotic classes of horizontal geodesics
are in bijection with points of $\del \Hyp^2$.  The set of all horizontal classes comprises 
the \emph{equator} of $\Gbound\HR$. 

These special subsets of $\Gbound\HR$ are artifacts of the product structure of $\HR$, and
are related to the fact that this geodesic boundary has the structure of a joint.  This means
the following. For each $(q, p^*) \in \del \HH^2 \times \del \RR$, $* = \pm$, there is a segment 
in $\Gbound \HR$ which consists of all classes $[(\gamma_1,\gamma_2)]$ where either
$[\gamma_1]=q$ or $\gamma_1$ is constant, and in addition, $[\gamma_2]=p^*$ or $\gamma_2$ is constant.
This segment is parametrized by the ratio of speeds $\rho = s_2/s_1\in \overline{\RR^+}$ and connects a point of 
the equator to one of the poles. It is called a \emph{Weyl chamber} and denoted $W^*(q)$, $* = \pm$. 
The set of all Weyl chambers is the union of two circles,  the set of midpoints of Weyl chambers, where $\rho=1$, 
is again the union of two circles, and can be identified with the Furstenberg boundary.

One fact which will play an important role later is that vertical translation acts trivially on $\Gbound\HR$.

\subsubsection{Relationship between $\Pcomp{\HR}$ and $\Gcomp{\HR}$}
These two compactifications are related in the following way.  Briefly, $\Gcomp{\HR}$
is obtained from $\Pcomp{\HR}$ by blowing up the corner and blowing down the top 
and bottom caps as well as the vertical cylinder. Similarly, to get to $\Pcomp{\HR}$ from
from $\Gcomp{\HR}$, we blow up the equator and the poles and blow down each of the joint
segments to points. 

In slightly more detail, the blowdown of the top and bottom caps of $\Pcomp{\HR}$ 
is the space where sequences $q_n \in \HR$ converging to any point of the top cap
in the product compactification are identified with one another, and hence correspond
to sequences converging to $p^+$ in the geodesic compactification.  In a similar way,
any point $(q,t) \in \del \HH \times \RR$ is mapped to the point $q$ on the 
equator of $\Gbound{\HR}$.  In the other direction, points in the corner $\del \HH \times \del \RR$
are stretched into the Weyl chambers, which thus captures the direction of approach to
infinity (i.e., the ratio of speeds) for geodesics $(\gamma_1, \gamma_2)$ which escape to infinity in both factors. 

From all of this, we see that points of $\Pbound{\HR}$ and $\Gbound{\HR}$ correspond
to different ways of distinguishing classes of diverging geodesics.

\section{Embedded minimal surfaces in $\HxR$} 

In this section we write out the PDE for minimal surfaces in $\HR$ in terms of
two natural graphical representations, and explain certain boundary regularity theorems. 
After that we recall a number of important examples of these surfaces which will enter our later constructions.

\subsection{Equations of horizontal and vertical minimal graphs}
In the following, use the half-space model for $\HH^2$ with coordinates $z=(x,y)$, $x > 0$, $y \in \RR$ and metric 
$(dx^2 + dy^2)/x^2$, as well as the corresponding coordinates $(x,y,t) \in \HR$ and metric 
\[
g = \frac{dx^2 + dy^2}{x^2} + dt^2
\]
on $\HR$. 

There are two obvious ways to represent a surface in $\HR$, namely as a graph over a horizontal slice $\HH \times \{0\}$, 
or as a graph over a vertical flat $F=\gamma\times \RR$ where $\gamma$ is a geodesic in $\Hyp^2$. In the former case, we write 
\[
\Sigma = \{ (z, u(z)): z \in \HH^2 \},
\]
while in the later, if $F=\{y=0\}$, then
\[
\Sigma = \{ (x,u(x,t),t): x>0, t\in\RR \}.
\]

The only nonvanishing Christoffel symbols in these coordinates are
\[
\Gamma^x_{xx} = -\frac1x, \ \ \Gamma^x_{yy} = \frac1x, \ \  \Gamma^y_{yx}=\Gamma^y_{xy} = -\frac1x.
\]


\subsubsection{Vertical graphs}
First suppose that $\Sigma$ is the graph of some function $u(x,y)$ defined over an open set $\calU \subset \HH^2$.
Then it is standard, see \cite{SaEarp-Toubiana} that the graph of $u$ is minimal if and only if
\[
\mathrm{div}^g \,
\left( \frac{\nabla^g u}{\sqrt{ 1 + |\nabla^g u|_g^2}} \right)=0. 
\]
where the divergence and gradient are taken with respect to the hyperbolic metric. Expanding in terms of 
the Euclidean metric $g_0$, this is the same as
\begin{equation}
\Delta_{g_0} u\left(1+x^2|\nabla^{g_0} u|_{g_0}^2\right) -x^2\sum_{i,j} u_{ij}u_i u_j -x u_x |\nabla^{g_0} u|_{g_0}^2 = 0,
\label{hormse}
\end{equation}
or finally, in upper half-space coordinates,
\begin{equation}
u_{xx}(1+x^2 u_y^2) + u_{yy}(1+x^2 u_x^2)  -2 x^2 u_{xy}u_x u_y  -x u_x(u_y^2+u_x^2) = 0.
\label{eq:horizontal}
\end{equation}

Solutions of this equation which are defined over all of $\HH^2$ are called {\it vertical graphs}; this parlance is perhaps
confusing inasmuch as the ends of these are horizontal.

Since \eqref{eq:horizontal} is uniformly elliptic, even near $x=0$, we deduce the following, 
see \cite[Theorem 15.11]{Gilbarg-Trudinger}.
\begin{prop}
Let $\calU$ be an open set in $\HHb$ which intersects $\del \HH^2$ in an open interval. Let $u(x,y)$
be a function defined on $\calU$ which satisfies \eqref{hormse}. If $u(0,y) \in \calC^{k,\alpha}$ for
any $k \in \mathbb N$ and $0 < \alpha < 1$, then $u$ is $\calC^\infty$ in $\calU \cap \HH^2$ 
and $\calC^{k,\alpha}$ on $\calU$ up to $x=0$. 
\label{bregver}
\end{prop}

\subsubsection{Horizontal graphs}
Now suppose that $\Sigma$ is a graph over a flat $F = \gamma \times \RR$, where
as before, $F = \{ y = 0\}$. We write this initially as $\{(x,y,t): y = v(x,t)\}$. 
Very similar calculations to the ones above show that this graph is minimal if and only if
\begin{equation}
v_{xx}(x^2+v_t^2)+v_{tt}(1+v_x^2) -2v_{xt} v_x v_t -x v_x(1+v_x^2)=0.
\label{eq:vertical}
\end{equation}
Solutions are called horizontal graphs (noting, however, that these surfaces have vertical ends). 

This equation is nondegenerate at $x=0$ only when $v_t(0,t) \neq 0$. However, if this condition were to hold, 
then the surface could also be written locally as a vertical graph, and its regularity near such 
portions of the boundary would therefore follow already from Proposition~\ref{bregver}. 

We restrict, therefore, to consideration  of the special case where $v(0,t) \equiv 0$ for $t$ lying in some
interval $\mathcal I$.   This excludes cases where $v_t(0, t_0) = 0$ at some isolated point $t_0$, 
as well as endpoints of vertical boundary regions, i.e.\ values $t_0$ such that $v(0,t) = 0$ for $t_0 - \epsilon < t \leq t_0$
and $v'(0,t) \neq 0$ for $t_0 < t < t_0 + \epsilon$.   If $|\mathcal I| > \pi$, then there is an immediate
a priori estimate for the behavior of $v$ near $x=0$ which is obtained by trapping $\Sigma$ between barriers
on either side. As barriers we use the `tall rectangles' which are described in the next section. (These are minimal
disks which intersect $\Pbound{\HR}$ in the union of two vertical lines of height greater than $\pi$ and
two circular arcs connecting their endpoints; these make a variable angle of contact along the vertical
portions of their boundary.)  This geometric argument yields a Lipschitz bound for $v$ as an immediate corollary.
\begin{prop}
Suppose that some portion of the boundary of the minimal surface $\Sigma$ is a vertical line $\mathcal I \subset \Pbound{\HR}$
with $|\mathcal I| > \pi$. Writing $\Sigma$ as a horizontal graph of a function $v$ defined on a vertical flat which
contains $\mathcal I$ in its boundary, then for each compact subinterval $\mathcal I'$ in $\mathcal I$ there exists
a constant $C > 0$ such that $|v(x,t)| \leq Cx$. 
\label{Lipbound}
\end{prop}
Proposition~\ref{theo:fillability} below provides a large class of examples of minimal surfaces which
have boundary containing a vertical line segment. 

Unfortunately it seems difficult to show that $\Sigma$ is actually smooth up to a vertical boundary segment. The reason is that the particular 
type of degeneracy in equation \eqref{eq:vertical}  only permits 
good regularity results in certain restricted cases, typically where the variable $t$ lies in a compact manifold with boundary 
(e.g.\ the circle) or else if we already know quite a bit more about the values of the function $v(x,t)$ along the curves 
$t = t_0$ and $t = t_1$ at the top and bottom of the interval $\mathcal I$.  We can, however, show that $v$ is {\it conormal}
at $x=0$, which is slightly weaker regularity statement.  We explain this now. 

First let us reparametrize $\Sigma$ by writing it as the exponential of $w \nu$, where $\nu = x \del_y$ is the {\it unit} normal to $F$.
It is not hard to check that $w = x v + \mathcal O(x^2)$. The Lipschitz bound in Proposition~\ref{Lipbound} shows that 
$\Sigma$ remains a bounded distance from $F$ even up to $x=0$, so $w(x,t)$ satisfies $|w| \leq C$. Next, change variables 
by setting $s = -\log x$, which means that we regard $w$ as a function of $(s,t)$. Thus $w$ satisfies the minimal surface 
equation $\mathcal M(w) = 0$ over an infinite region $\mathcal S = [s_0, \infty) \times \mathcal [t_0, t_1]$. 

In these coordinates, the minimal surface operator $\mathcal M$ is a quasilinear elliptic operator with uniformly 
bounded coefficients, see \cite{Gilbarg-Trudinger}, and it follows from classical estimates (applied on any ball of
some fixed radius $r_0$ in the $(s,t)$ coordinates) there that the function $w$ and all its derivatives
are uniformly bounded, i.e.,  $|\del_s^p \del_t^q w| \leq C_{p,q}$ for every $p, q \in \mathbb N$.  Equivalently, since
$\del_s = x\del_x$, $|(x\del_x)^p \del_t^q w| \leq C_{p,q}$ for all $p, q$. However, this collection of estimates is just
the definition of what is known as {\it conormal} regularity of order $0$ at the boundary $x=0$. The graph function
$v$ is conormal of order $1$, i.e., $|(x\del_x)^p \del_t^q v| \leq C_{p,q}x$ for all $p, q$. We summarize all of this in the
\begin{prop}
Let $\Sigma$ be a complete minimal surface which contains a vertical interval $\mathcal I$ in its asymptotic 
boundary. Write $\Sigma$ as a horizontal graph, with graph function $v(x,t)$, and assume that $|v(x,t)| \leq Cx$
(this condition is automatic by the barrier construction if $|\mathcal I| > \pi$). Then for all $p, q \in \mathbb N$, 
\[
|(x\del_x)^p \del_t^q v(x,t)| \leq C_{p,q} x,
\]
i.e., $v$ is conormal of order $1$. 
\end{prop}

Now consider the problem of determining whether $\Sigma$ is actually smooth up to $\Pbound{\HR}$; this corresponds to the 
assertion that $w$ has an asymptotic expansion $ \sim \sum_{j \geq 0} e^{-js} w_j(t)$ where each $w_j \in \calC^\infty$. 
Suppose that we are in the even more restricted case where $\Sigma$ is locally trapped between portions of two flats which 
intersect along $\mathcal I$. This means that $|w| \leq C e^{-s}$. We can then study the minimal surface equation for $w$ 
perturbatively; in other words, we use the Taylor expansion of $\mathcal M$ at $0$, which can be written as
\[
Lw = Q(w, \nabla w, \nabla^2 w),
\]
where $L = \del_s^2 + \del_t^2 - 1$ is the Jacobi operator along $F$ and $Q$ is a quadratic remainder term.
The same classical local elliptic theory (in the $(s,t)$ coordinates) shows that all higher derivatives of $w$ satisfy the
same bounds as $w$ itself, i.e., $|\del_s^p \del_t^q w| \leq C_{p,q} e^{-s}$ uniformly in any semi-infinite 
strip $\mathcal S' \subset \mathcal S$. Relabelling $\mathcal S'$ as $\mathcal S$, we study this by first considering 
the inhomogeneous problem $Lw = f$ where $|f| \leq C e^{-2s}$ (along with similar estimates for its higher derivatives). 
Simplifying even further, consider the homogeneous problem 
\[
Lw = 0,  \quad w(t_0, s) = w_0(s),\ w(t_1, s) = w_1(s).
\]
If we happen to know that $w_0, w_1$ admit asymptotic expansions in integer powers of $e^{-s}$ as $s \nearrow \infty$,
then it is straightforward to show that the same is true for $w(s,t)$ for $t \in \mathcal I$. A small modification
of this argument gives the same conclusion for the nonlinear equation. However, our a priori information is only that
$|w_0|, |w_1| \leq Ce^{-s}$, and this is not enough information to reach the desired conclusion, even in a slightly
smaller strip. In fact, it is not hard to show that if $w_0$ and $w_1$ decay exponentially but do not have such smooth 
expansions, then neither does the solution $w$.   However, it is unclear whether such a phenomenon can happen
in our setting, i.e., when $\Sigma$ is a complete minimal surface, and it is entirely possible that $\Sigma$ is actually
smooth up to vertical boundaries, but this remains an open question.

Notice that we are not making any claim about the regularity of this graph at the horizontal caps
$\HH^2 \times \{\pm \infty\}$. We discuss this point later. 

\subsection{Examples} 
We review here the basic examples of properly embedded minimal surfaces in $\HR$. 
These give some intuition about boundary behavior of more general surfaces of this type,
and some of these will also be used as barriers in constructions below. 

\subsubsection{Horizontal disks and vertical planes}  
The simplest examples of properly embedded minimal surfaces are the ones which respect the product
structure of $\HR$.  These are the horizontal disks
\[
\HH^2 \times \{a\},
\]
for any $a \in \RR$, and the vertical planes
\[
F = \gamma \times \RR,
\]
where $\gamma$ is a geodesic in $\HH^2$. Note that each of these is not only minimal, but totally geodesic.
The horizontal planes have Gauss curvature $K \equiv -1$, while the flats have $K \equiv 0$. 

\subsubsection{Tall rectangles}
The next example is a two-parameter family of minimal surfaces in $\HR$,  each element of which
has asymptotic boundary which is a finite `rectangle' in the vertical boundary of $\Pbound{\HR}$. 
Limiting elements of this family are semi-infinite or infinite rectangles; the semi-infinite ones have
asymptotic boundary which includes an entire geodesic in one of the two horizontal components
of $\Pbound\HR$, while the infinite ones are simply flats, $F = \gamma \times \RR$.  This family was initially described 
by Hauswirth \cite{Hauswirth} and independently Sa Earp and Toubiana \cite{SaEarp-Toubiana}. These surfaces were first
used as a very useful set of barriers in \cite{Mazet-Rodriguez-Rosenberg_err}. 
The name `tall' is due to Coskunuzer \cite{Coskunuzer} and refers to the fact that these only exist when their height is 
greater than $\pi$. 

Fix any arc $c\subset \partial \Hyp^2$ and denote its endpoints by $q_1$ and $q_2$; choose $a, b\in \RRb$
with $a < b$ and $\ell:=b-a > \pi$. There is a unique minimal surface $H=H(c,a,b)\subset \HR$ which is
a horizontal graph over the rectangle $c \times [a,b]$. This surface $H$ has the following properties.
Let $\gamma$ be the geodesic in $\HH^2$ which terminates at $q_1$ and $q_2$. Then 
$H$ is contained in the half-space bounded by $F = \gamma \times \RR$. It is invariant
under the one-parameter group of isometries of $\HH^2$ which fix $q_1$
 and $q_2$; these are isometries of hyperbolic type, and extend in a natural way to isometries of $\HR$ leaving each
level $\{t = \mathrm{const}\}$ fixed.   

Because of this isometry invariance, it is not hard to see that the intersection of $H$ with
any horizontal slice $\HH^2 \times \{t\}$ is a curve in that copy of $\HH^2$ which is equidistant
from the geodesic $\gamma$. Denoting its distance from $\gamma$ by $\rho(t)$, then
$\rho(t) = \infty$ at $t = a, b$, corresponding to the fact that $H \cap (\HH^2 \times \{a\}) = c \times \{a\}$,
and similarly at $t=b$.  For the limiting case $a = -\infty$, $b = \infty$, we can
take $\rho(t) \equiv 0$, so that $H_{c, -\infty, \infty}$ is the flat $F = \gamma \times \RR$. 
When $-\infty < a$ but $b = \infty$, then $\rho(t)$ is decreases monotonically from
$\infty$ to $0$ as $t$ increases from $a$ to $\infty$.  When $a$ and $b$ are both finite,
then $\rho(t)$ is proper and convex on the interval $(a,b)$, and obviously symmetric
around the midpoint $t = (a+b)/2$.   Finally, this minimum value of $\rho$ tends 
to $\infty$ as the overall height $\ell$ of the boundary rectangle decreases to $\pi$. 

All of this can be deduced from an explicit expression for the function $\rho(t)$ involving 
integrals, see \cite{SaEarp-Toubiana}. 
Figure \ref{fig:barriers} illustrates some members of this family of surfaces. 

If $a$ and $b$ are finite, then the asymtotic boundary of $H(c, a,b)$ in the product
compactification is the rectangle with four segments:
\[
 \Pbound H(c, a,b) = (c \times \{a,b\})\cup (\{q_1, q_2\} \times [a,b]).
\]
Similarly, 
\[
 \Pbound H(c, a, \infty) = (c \times \{a\}) \cup ( \{q_1, q_2\} \times [a, \infty]) \cup (\gamma \times \{\infty\}),
\]
and
\[
 \Pbound H(c, -\infty, \infty) = (\{q_1, q_2\}) \times \RRb) \cup (\gamma \times \{\pm \infty\}).
\]

In the geodesic compactification, on the one hand
 when $a$ and $b$ are finite, $\Gbound H(c, a,b)$ is
simply the arc $c$ on the equator. On the other hand 
$\Gbound H(c, a, \infty)$ is the union of $c$ and of the two
Weyl chambers connecting the endpoints of $c$ to the north pole $p^+$, 
$\Gbound H(c, -\infty, b)$
is the union of $c$ and of the two Weyl chambers connecting points of $c$ 
both to $p^+$ and to $p^-$, and $\Gbound H(c,-\infty,\infty)$ is the union 
of the four Weyl chambers having $q_1$ or $q_2$ as endpoints (including 
$p^\pm,q_1,q_2$).

\begin{figure}
\labellist
\small \hair 5pt
\pinlabel $\ell\in(\pi,+\infty)$ [l] at 220 185
\pinlabel $a$ [l] at 257 140
\pinlabel $b$ [l] at 257 227
\pinlabel $c$ [t] at 141 118
\pinlabel $\ell=+\infty$ [l] at 535 203
\pinlabel $a$ [l] at 574 140
\pinlabel $b=+\infty$ [l] at 574 255
\pinlabel $c$ [t] at 452 118
\endlabellist
\centering
\includegraphics[scale=0.6]{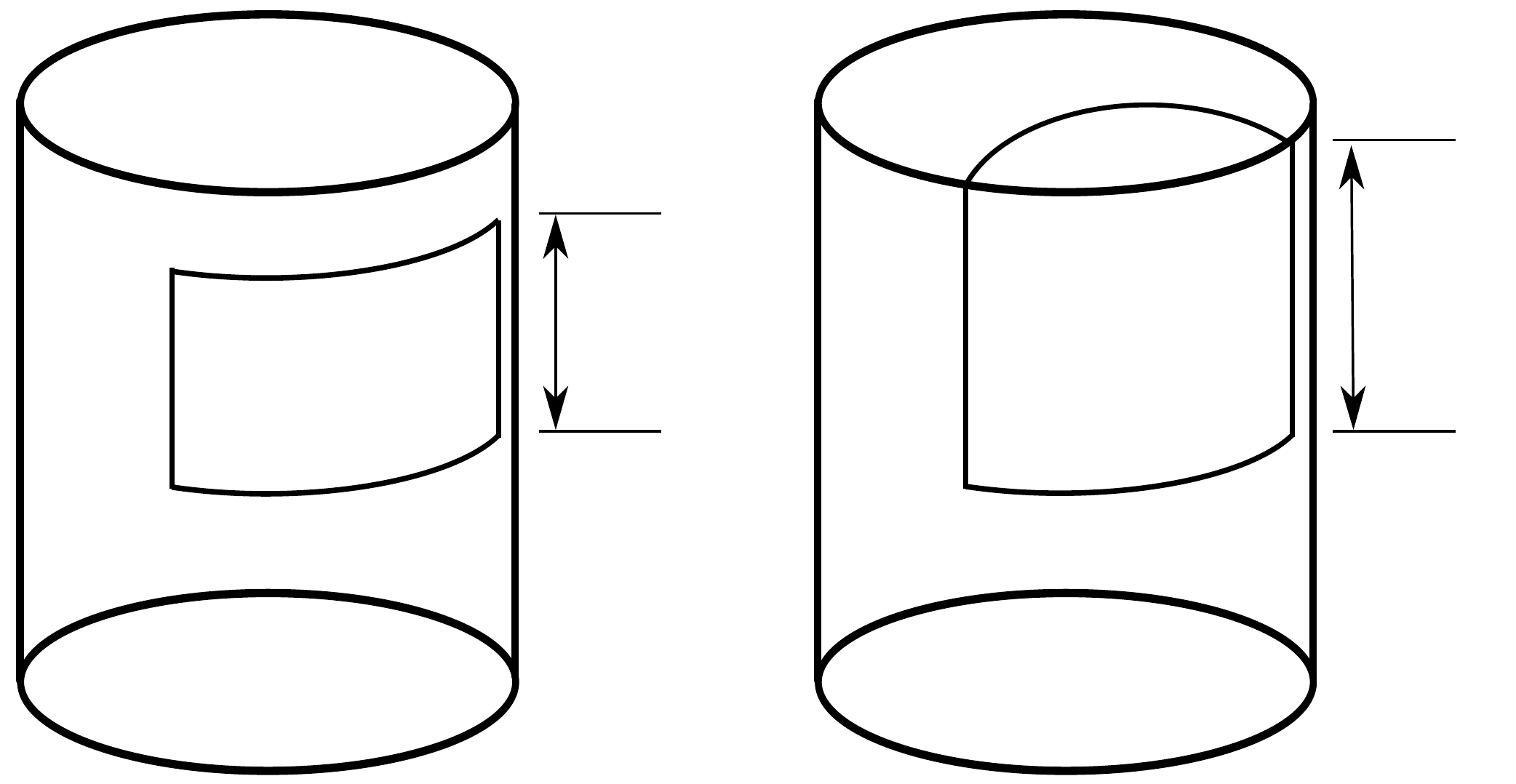}
\caption{The boundary of tall rectangles in the product compactification.}
\label{fig:barriers}
\end{figure}
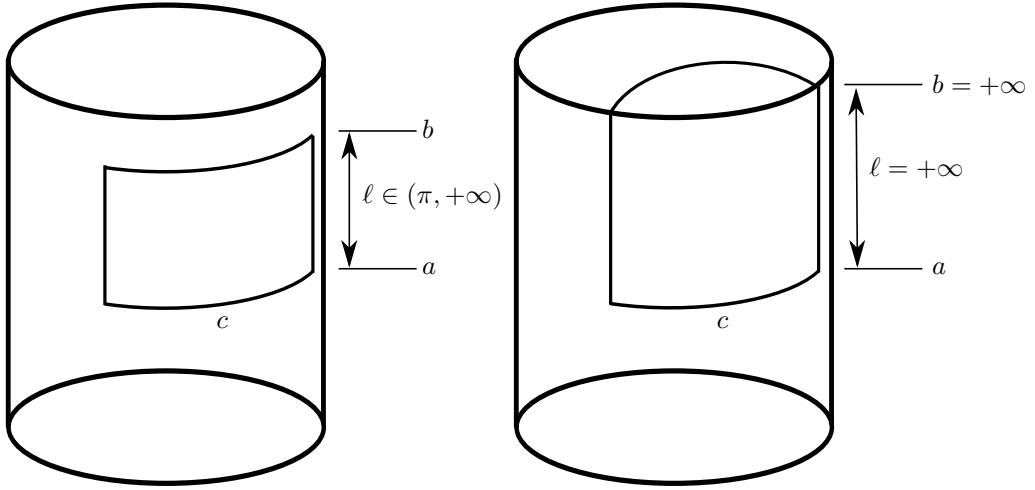

The expressions for these surfaces are written out in the disk model in \cite{SaEarp-Toubiana}, but  
we find it more useful here to write them in 
the half-plane model. We do so using the minimal surface equation for vertical graphs. 

Use coordinates $(x,y,t) \in \RR^+ \times \RR\times \RR$ and suppose that
the boundary arc is the segment between $q_1=(0,1)$ and $q_2=(0,-1)$. Also let $a=0$. 
The circular arcs with endpoints $q_1,q_2$, which correspond to curves equidistant from the geodesic
connecting $q_1$ and $q_2$, are described by
\[
s(x,y):=\frac{1-(x^2+y^2)}{2x}=\mbox{const.}
\]
We seek graphs $t=u(x,y)$ with $u(x,y)=f(s(x,y))$.  When $b < \infty$, $u$ is multi-valued,
since the entire surface is a bigraph over a concave region in $\HH^2$ bounded by an
equidistant curve. When $b=\infty$, the surface is a single-valued graph 
over a half-space in $\HH^2$. 

One computes 
\begin{align*}
s_x  &  = -1-\frac{s}{x},  &  s_y  &  = -\frac{y}{x}  \\
u_x  &  = (-1-\frac{s}{x})f'(s),  &  u_y  &  = -\frac{y}{x}f'(s)
\end{align*}
Then
\begin{align*}
u_{xy}  &  = \frac{y}{x^2}f'(s)+\frac{y}{x}(1+\frac{s}{x})f''(s)\\
u_{xx}  &  = \frac1x(1+2\frac{s}{x})f'(s) + (1+\frac{s}{x})^2 f''(s) \\
u_{yy}  &  = -\frac1x f'(s) + \frac{y^2}{x^2} f''(s).
\end{align*}
Plugging this into \eqref{eq:horizontal} and using that $y^2+x^2+2xs=1$, we obtain
\begin{equation}
0 = 2sf' + s(1+s^2)f'^3 + (1+s^2) f''. 
\end{equation}
This is a Bernoulli equation in $f'$, which has solutions
\[
f'(s) = \frac{-1}{\sqrt{C-1+(2C-1)s^2+C s^4}}.
\]
where $C\in(0,1)$. This is impossible to integrate in explicit terms, except when $C=1$, which 
corresponds to $b=\infty$, and in that case 
\[
f(s)=\argtanh\left(1+s^2\right)^{-\frac12}.
\]

\subsubsection{Horizontal and vertical catenoids}
As a final set of examples, we list the horizontal and vertical catenoids.   

The horizontal catenoids $\calC_H(p, a,b)$ were first described by Pedrosa and Ritor\'e \cite{Pedrosa-Ritore},
but independently and more explicitly by Nelli and Rosenberg \cite{Nelli-Rosenberg,Nelli-Rosenberg_err}.
Each one of these is  a surface of rotation in $\HR$ around a vertical axis $\{p\} \times \RR$, where the point 
$p$ can be arbitrary in $\HH^2$, and contained in the slab $a \leq t \leq b$. 

Assuming that $p = 0$ and that the catenoid is symmetric with respect to $(t=0)$, 
i.e. $a=-b$, a parametrization of $\calC_H(p,-b,b)$ is
\[
\{ (r(t) \cos\theta, r(t) \sin \theta, t)\},  \quad |t| < b, \ \theta \in S^1,
\]
where $r(t)$, the radial coordinate in the disk model of $\HH^2$, satisfies 
\[
r' = \pm\sqrt{Cr^2-\frac{1+r^4}{4}};
\]
here $C$ is a constant determined by the height $2b$ \cite{Nelli-Rosenberg}. 
These solutions exist if and only if $C>1/2$, which corresponds to the height limitation $2b < \pi$.
If $\varphi: \HH^2 \to \HH^2$ is any isometry, then its extension $(\varphi, \mbox{Id})$ to $\HR$ acts on the
space of horizontal catenoids, sending $\calC_H(p, a,b)$ to $\calC_H( \varphi(p), a,b)$.  If $\varphi_\lambda$ is
a family of hyperbolic isometries which fixes two points $q_\pm \in \del \HH^2$, then the limit of $\calC_H( \varphi_\lambda(p), a,b)$ 
is the union of two horizontal hyperbolic planes $\HH^2\times\{a,b\}$; the neck has disappeared at infinity.
There is a different limit if one lets $\lambda \to \infty$ and simultaneously $b-a \nearrow \pi$. This is called
Daniel's surface \cite{Daniel}, and is a disk which has boundary along the union of two circles at distance
$\pi$ from one another and a straight line connecting these circles. Daniel's surface can also be obtained
as the limit of tall rectangles as $b-a \searrow \pi$ and the arc $c$ converges to the whole circle. 

The vertical catenoids $\calC_V$ are parametrized, by contrast, by pairs of geodesics $\gamma_\pm \subset \HH^2$
such that $\mbox{dist}(\gamma_+, \gamma_-)$ is less than some critical value $\eta_0$. Each of these geodesics 
determines a flat $F_\pm = \gamma_\pm \times \RR$,
which is a vertical plane, and $\calC_V( \gamma_+, \gamma_-)$ is the annular surface which is asymptotic to this
pair of flats.  These were first constructed in \cite{Morabito-Rodriguez} and \cite{Pyo}, but see also \cite{Martin-Mazzeo-Rodriguez} 
for a  careful explanation of their geometrical properties.

In contrast to the other surfaces described above, the horizontal and vertical catenoids have disconnected
boundaries.

\section{Minimally fillable curves on the product boundary}
We now describe broader classes of properly embedded minimal surfaces in $\HR$ which generalize the
examples above in various ways.  Much of this summarizes previously known results, but we present
a few new results too.  

The general motivation for these questions here is the basic one, to determine which curves in $\Pbound\HR$ 
or $\Gbound\HR$ occur as boundaries of properly embedded minimal surfaces $\Sigma\subset\HR$.  We call 
any such curve \emph{minimally fillable}. By curve, we implicitly mean a finite union of disjoint 
\emph{Jordan curve}, i.e., of topological embedding of the circle; the word \emph{arc} signifies a continuous image of an interval.

We also discuss the more restrictive problem of characterizing
curves $\sigma$ for which the filling is not only minimal but area-minimizing, which means that any
compact portion of this surface is absolutely area-minimizing. 

\subsection{Barriers and minimally fillable curves}
Most of the existence theorems rely on some version of the following folklore result: 
\begin{prop} Let $\sigma\subset\Pbound\HR$ be a curve, and assume that every
$p\in \Pbound\HR\setminus\sigma$ can be separated from $\sigma$ by the boundary $\Pbound \Sigma_p$ of
a properly immersed minimal surface $\Sigma_p$. Then $\sigma$  is minimally fillable.
\label{folklore}
\end{prop}
The proof of this proceeds as follows. Let $\sigma_j$ be a sequence of curves in $\HR$ which approach $\sigma$,
and for each $j$, let $\Sigma_j$ be a solution of the Plateau problem (or indeed any other minimal surface) with
$\del \Sigma_j = \sigma_j$. The key step in showing that the $\Sigma_j$ converge to a solution of our problem
is to ensure that $\Sigma_j$ does not leave every compact set as $j \to \infty$. However, regarding $\Sigma_j$
in $\Pcomp{\HR}$, we see that it is impossible for this sequence to have any limit points $p$ which do not lie on 
$\sigma$, since the surfaces $\Sigma_p$ serve as barriers which prevent $\Sigma_j$ from having $p$ as a limit point. Since $\sigma$ is a curve,
any sequence of surfaces escaping all compacts
and whose boundaries approach $\sigma$ must have an accumulation point
outside $\sigma$, completing the proof.

A recent result by Coskunuzer \cite{Coskunuzer} settles part of the problem of characterizing curves which are contained
in the vertical boundary of $\Pbound\HR$ and which are fillable by {\it minimizing} rather than just \emph{minimal} surfaces. 
To state his result, we recall his terminology that a curve $\sigma \subset \del \HH^2 \times \RR$ 
is called {\it tall} if $(\del \HH^2 \times \RR) \setminus \sigma$ is a union of tall rectangles, i.e., 
rectangles of the form $c \times [t_1, t_2]$, where $c$ is an arc in
$\del \HH^2$ and $t_2-t_1 > \pi$.  Next, define the height $h(\sigma)$ of the curve $\sigma$
to be the infima of lengths of the bounded components of $(\{\theta_0\} \times \RR) \cap \sigma$. He proves the following
\begin{prop}[\cite{Coskunuzer}]
A (possibly disconnected) curve $\sigma \subset \del \HH^2 \times \RR$ with $h(\sigma) \neq \pi$ is the boundary 
of a properly embedded
area-minizing surface $\Sigma$ if and only if $\sigma$ is tall. 
\label{Cosprop}
\end{prop}
The proof of existence when $\sigma$ is tall is much the same as above: take a sequence of curves $\sigma_i \subset \HR$ 
converging to $\sigma$, and find an area-minizing surface $\Sigma_i$ with $\del \Sigma_i = \sigma_i$ for each $i$. 
One then uses tall rectangles as barriers to show that some subsequence of the $\Sigma_i$ converge to a solution of the problem. 
Nonexistence when $\sigma$ is short is proved by a cut and paste argument which uses strongly the fact that one
is one is seeking a minimizing filling $\Sigma$. 

This result does not guarantee that the surface $\Sigma$ has only one component; for example, if $\sigma$ consists 
of a pair of horizontal circles which are sufficiently far apart, then the (unique) minimizing surface bounded by them
is the pair of horizontal disks.  

The criteria for existence of minimal fillings are clearly different: to return to the example where $\sigma
 = S^1 \times \{t_1, t_2\}$, suppose now that $t_2 - t_1 < \pi$ so that $\sigma$ is short. Then there is no 
minimizing filling, but on the other hand there is a minimal catenoid $\Sigma$ with $\Pbound \Sigma = \sigma$. 

Many interesting questions remain open.  We discuss a few of these after presenting several types of existence results. 

\subsection{Vertical and horizontal graphs}
The direct generalization of the family of horizontal disks $\HH^2 \times \{a\}$ are the vertical graphs.
It is reasonable to expect a general existence theorem for solutions of the minimal surface equation
for vertical graphs over $\HH^2 \times \{0\}$, based on the well-known solvability of the 
asymptotic Plateau problem for harmonic functions on $\HH^2$, and this is indeed the case. 
Let $u_0: S^1 \to \RR$ be any $\calC^0$ function.  Nelli and Rosenberg \cite{Nelli-Rosenberg} proved that 
there exists a solution $u: \HH^2 \to \RR$ to \eqref{eq:horizontal}, such that $u$ extends
continuously to $\overline{\HH^2}$ with $\left. u \right|_{\del \HH^2} = u_0$. This solution is unique. 
By Proposition~\ref{bregver}, $u \in \calC^{k,\alpha}(\overline{\HH^2}) \cap \calC^\infty(\HH^2)$ 
if $u_0 \in \calC^{k,\alpha}(S^1)$. 

The analogous problem of finding solutions which are graphs over vertical flats $F = \gamma \times \RR$ 
is less developed.  One point is that the possibly minimally fillable boundaries are more constrained. 
We have already quoted Coskunuzer's result, Proposition~\ref{Cosprop}, but an earlier 
result by Sa Earp and Toubiana \cite{SaEarp-Toubiana} presents a rather odd restriction, that if $\sigma$ has
a `thin tail', then $\sigma$ has no minimal filling. By definition a thin tail in $\sigma$ is an open subarc $c \subset \sigma$
which lies entirely on one side of some vertical line $\{p\} \times \RR$, $p \in \del \HH^2$ except
at some interior point or segment of the arc where it intersects this line, and such that $c$ lies within a horizontal slab 
$\del \HH^2 \times [t_1, t_2]$ with $t_2 - t_1 < \pi$. 

Another restriction concerns the behavior of horizontal graphs near the top and bottom caps
$\HH^2 \times \{\pm \infty\}$. 
\begin{prop}
Suppose that $\sigma\subset\Pbound\HR$ is a curve, and denote by 
\[
\sigma^\pm=\{p\in\HH^2 \mid (p,\pm\infty)\in \sigma\}
\]
the intersections of $\sigma$ with the upper and lower caps. 
If $\sigma$ is minimally fillable, then $\sigma^\pm$ are unions 
of disjoint geodesics in $\HH^2$. 
\end{prop}
\begin{proof}
Let $\Sigma$ be an embedded minimal surface with $\sigma=\Pbound \Sigma$. Setting $T_s(x,y,t)=(x,y,t-s)$,
then write $\Sigma_s=T_s(\Sigma)$.  

By definition of $\Pcomp{\HR}$, if $p$ is any point in $\sigma^+$ and $s_j$ any sequence
which tends to infinity, then there exist $p_j \in \HH^2$ such that $(p_j,0) \in \Sigma_{s_j}$ and 
$p_j \to p$.  Clearly, $\Sigma_{s_j}$ converges to a complete minimal surface $\Sigma^+$ which contains $(p,0)$,
and since this is true for any $p \in \sigma^+$ for the same sequence $s_j$, the surface $\Sigma^+$ contains 
$\sigma^+$.  Replacing the sequence $s_j$ by $s_j + a$ for any fixed $a \in \RR$, we deduce that
$\Sigma^+$ equals $\sigma^+\times\RR$.  It is straightforward to check that this product is minimal
if and only if $\sigma^+$ be a union of disjoint geodesics.

The same conclusion obviously holds for $\sigma^-$.
\end{proof}

The two restrictions we have now seen are equivalent to the following. Using upper half-plane coordinates on 
$\HH^2$, if $\Sigma$ is a minimal graph over $F = \{y = 0\}$ with graph function $v(x,t)$, $v(\cdot,t)$ tends to 
a hyperbolic geodesic as $|t| \to \infty$, and moreover, the restriction $v(0,t)$ cannot have any local maxima or minima 
in any interval $[t_1, t_2]$ with $t_2 - t_1 < \pi$ (this is the nonexistence of thin tails).  These suggest that it 
may not be easy to formulate sharp conditions for the existence of minimal horizontal graphs. 

There are, however, some interesting nontrivial solutions.
\begin{prop}\label{theo:fillability}
Suppose that $\sigma\subset\Pbound\HR$ is a curve such that each of the arcs $\sigma^\pm$ at the top and bottom
caps are either empty or a single complete geodesic in $\HH^2$, and in addition, for each vertical line 
$L \subset \Pbound\HR$, the components of $L \setminus\sigma$ are intervals of length greater than $\pi$,
then $\sigma$ is minimally fillable.
\end{prop}
As an example, $\sigma$ could be the union of two geodesics on the upper and lower caps
along with two arcs lying along $\del \HH^2 \times \RR$ which are monotone with respect
to $t \in \RR$ and which connect the respective endpoints of these geodesics, and which do 
not have thin tails, see Figure \ref{fig:twisted}. 
\begin{proof}
We apply Proposition \ref{folklore}. To do so, we must show that if $p\in\Pbound \HR\setminus\sigma$, 
then there is a curve which separates $p$ from $\sigma$ and which is minimally fillable. 

If $p \in \del\HH^2\times\RR$, let $L$ be the vertical line containing $p$ and $\{q\}\times (a,b)\subset L$ 
the connected component of $L\setminus\sigma$ containing $p$. By hypothesis $b-a>\pi$, so there is a small arc
$c\ni q$ and a tall rectangle $\Pbound H(c,a+\varepsilon,b-\varepsilon)$ separating $p$ from $\sigma$, as desired.

Next, assume $p \in \overline{\HH^2}\times\{\infty\}$, say. If $\sigma^+$ is empty, we can separate $p$ from 
$\sigma$ by a horizontal circle $\Pbound \HH^2\times\{t\}$, which is fillable by the horizontal disk. Otherwise,
let $\gamma$ be a geodesic in $\HH^2$ separating $p$ from $\sigma^+$. If $c$ is the arc of $\del\HH^2$ which 
joins the endpoints of $\gamma$ without meeting $\sigma^+$, then for sufficiently large $a$, the 
the semi-infinite tall rectangle $\Pbound H(c,a,\infty)$ separates $p$ from $\sigma$.
\end{proof}

\begin{figure}[Ht]
\labellist
\small \hair 3pt
\pinlabel $p$ [r] at 117 250
\pinlabel $p$ [r] <2pt,0pt> at 20 135
\pinlabel $\sigma$ [l] at 82 119
\pinlabel $\sigma^+$ <0pt,-3pt> [br] at 106 270
\pinlabel $\sigma^-$ [bl] at 126 33
\endlabellist
\centering
\includegraphics[scale=.5]{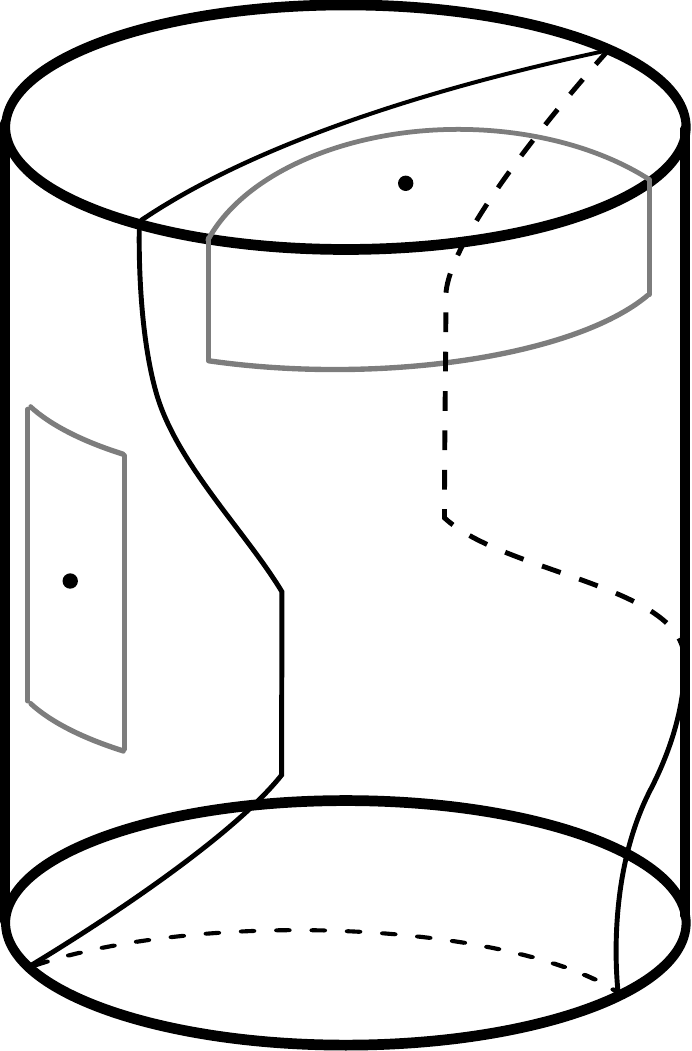}
\caption{The curve $\sigma$ is fillable, since the points $p\notin\sigma$ can be enclosed by tall rectangles.}
\label{fig:twisted}
\end{figure}
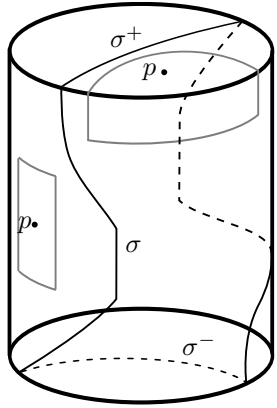

Note that the examples given by Proposition \ref{theo:fillability} are different than the various generalized 
helicoids that have been previously constructed. Indeed, there is greater flexiblity here, e.g. the winding of the
vertical arcs around $\partial\HH^2\times \RR$ is relatively unconstrained. while on the other hand, 
helicoids can cut the vertical lines into short intervals. One important difference is that the boundary of 
helicoids in $\Pbound \HR$ contains the entire caps at $t = \pm \infty$; the examples here only make a 
finite number of turns, which avoids this issue.

\subsection{Contractible curves on the vertical boundary} 
Nullhomotopic curves $\sigma$ on $\del \HH^2 \times \RR$ exhibit rather different minimal fillability properties.  
The tall rectangles and their fillings discussed above are basic models in this class.  We have already described Coskunuzer's 
result, characterizing which of these have minimizing fillings, and also pointed out the basic obstruction that any curve with 
a thin tail has no minimal filling. 

We describe here another class of examples, which we call butterfly curves. These illustrate that in this setting too it
may be difficult to fully characterize the fillable nullhomologous curves. Our discovery of these curves below was
one of the starting points of the present work; however, in the intervening time, Coskunuzer's paper appeared and it
contains a similar class of examples. 

Fix positive reals $\ell<\pi<L$, numbers $a,b$ such that $a<b$, $b+\ell<a+L$ and four cyclically ordered points 
$q_1,\dots,q_4$ on $\partial\Hyp^2$. We define the butterfly curve associated to these numbers to be the 
concatenation $\sigma$ of \begin{itemize}
\item the vertical segments $\{q_1\}\times[a,a+L]$ and $\{q_4\}\times[a,a+L]$; 
\item the four horizontal arcs 
\[
\overset{\frown}{q_1 q_2}\times \{a, L\}, \quad \overset{\frown}{q_3 q_4} \times \{a, L\};
\]
\item the four vertical segments 
\[
\{q_2, q_3 \}\times[a,b], \quad \{q_2, q_3 \}\times[b+\ell, a+L]; 
\]
\item the two horizontal arcs 
\[
\overset{\frown}{q_2 q_3} \times \{ b, b+\ell \}.
\]
\end{itemize}
See Figure \ref{fig:butterfly} for an illustration.  Observe that this curve does not satisfy the criterion in 
Theorem \ref{theo:fillability} because the vertical distance between the two last horizontal arcs is less than $\pi$.

\begin{figure}[tb]
\labellist
\small \hair 4pt
\pinlabel $q_1$ [t] <1pt,0pt> at 61 115
\pinlabel $q_2$ [t] at 121 113
\pinlabel $q_3$ [t] <2pt,0pt> at 156 120
\pinlabel $q_4$ [t] <-1pt,-2pt> at 190 136
\pinlabel $a$ [r] at 61 116
\pinlabel $a+L$ [r] at 61 200
\pinlabel $b$ [r] at 121 143
\pinlabel $b+\ell$ [r] at 121 174
\pinlabel $\ell$ [l] at 160 164
\pinlabel $L$ [l] at 213 182
\endlabellist
\centering
\includegraphics[scale=0.6]{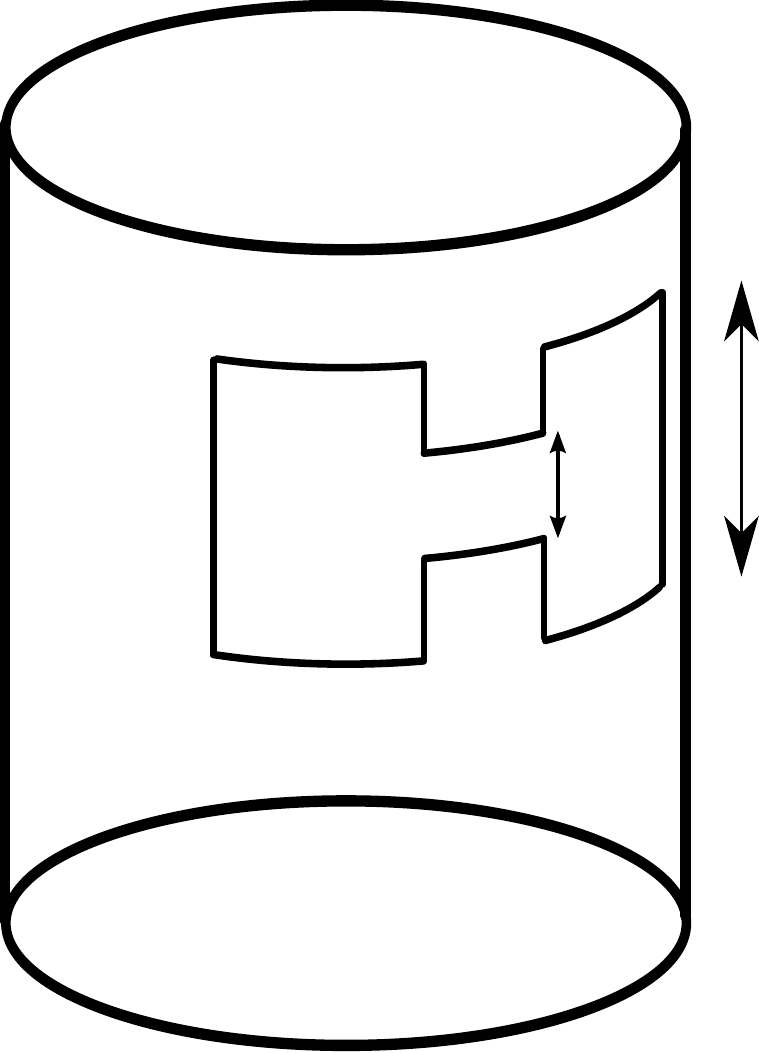}
\caption{An $(\ell,L)$ butterfly curve.}
\label{fig:butterfly}
\end{figure}
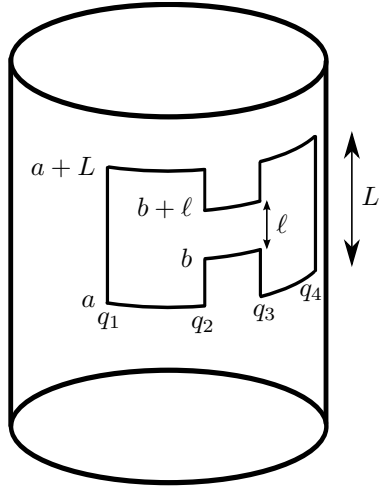

\begin{theo}\label{theo:butterfly}
For every $\ell\in(0,\pi)$, there exists an $L>\pi$ and a butterfly curve with the
parameters $(\ell,L)$ which is minimally fillable.
\end{theo}
\begin{proof}
We use a slight variant of Proposition~\ref{folklore}, using a composite barrier formed by two tall rectangles 
and a portion of a catenoid, the interior boundary of which lies in the tall rectangles.  This barrier
prevents convergence of approximations to points on $\del \HH^2 \times \RR$ inside $\sigma$;
the tall rectangles alone provide similar barriers for points outside $\sigma$.

So, let $\gamma_1$ and $\gamma_2$ be the geodesics in $\HH^2$ which connect $q_1$, $q_2$, and $q_3$, $q_4$, respectively, 
Next, choose a horizontal catenoid $C$ of height $\ell$, the projection of which onto $\RR$ is the interval $[b,b+\ell]$ and
such that the disk $D \subset \HH^2$ which is the complement of the projection of $C$ to the $\HH^2$ factor intersects 
both $\gamma_1$ and $\gamma_2$, see Figure \ref{fig:composite}.
Note that such choices can be made as soon as the hyperbolic
diameter of $D$ is larger than the distance between $\gamma_1$ and
$\gamma_2$, so that there is a trade-off between the parameter $\ell$
and the cross-ratio of $q_1,q_2,q_3,q_4$: to realize a small $\ell$,
one has to start with $q_1$ close to $q_4$ or $q_2$ close to $q_3$.

Finally, let $A$ be the angular section with the same
center as $D$, and such that the interior boundary of $C \cap (A\times\RR)$ 
lies on the same side of the flats $\gamma_1\times\RR$ 
and $\gamma_4\times\RR$ as the arcs $\overset{\frown}{q_1q_2}$ and $\overset{\frown}{q_3q_4}$, respectively.

\begin{figure}[tb]
\labellist
\small \hair 3pt
\pinlabel $q_1$ [br] at 45 198
\pinlabel $q_2$ [tr] at 88 6
\pinlabel $q_3$ [t] at 112 4
\pinlabel $q_4$ [bl] at 202 171
\pinlabel $\gamma_1$ [bl] at 70 161
\pinlabel $\gamma_4$ [tl] at 171 147
\pinlabel $D$ at 106 128
\pinlabel $A$ at 141 47
\endlabellist
\centering
\includegraphics[scale=0.6]{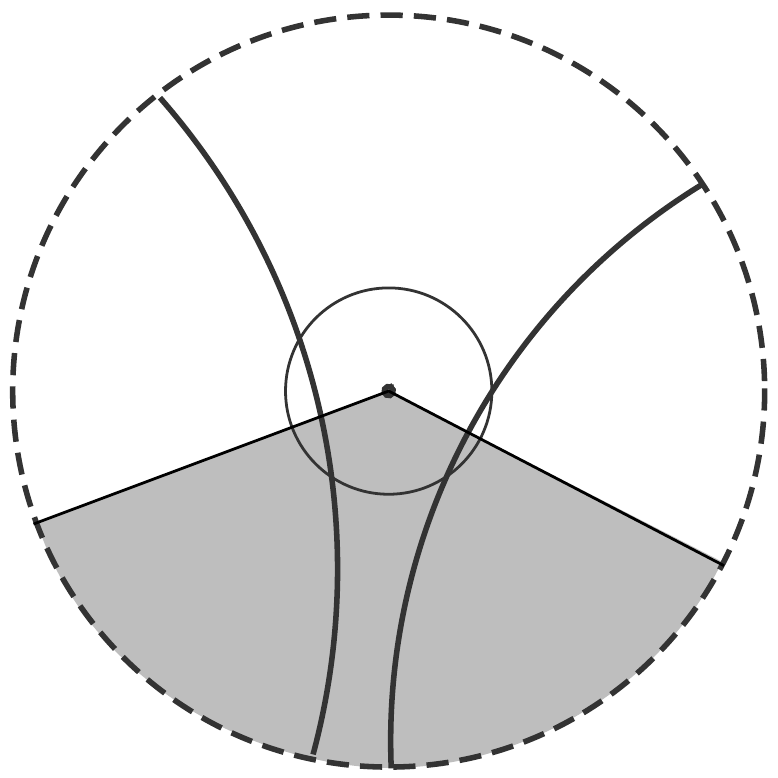}
\caption{Construction of the composite barrier, viewed projected in $\Hyp^2$.}
\label{fig:composite}
\end{figure}
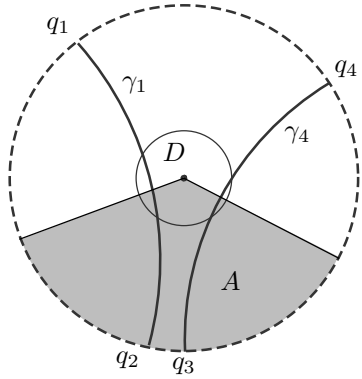

Now choose $L$ sufficiently large so that the tall rectangles $H, H'$ of height $L$ with the
same plane of symmetry as $C$ approximate these flats closely enough so that these too 
separate the interior boundary of $C'$ from the center of $D$. 
This is always possible, but depending on the choices of $\ell$ and 
$q_1,q_2,q_3,q_4$, the diameter of $D$ can be barely
larger than the distance between $\gamma_1$ and $\gamma_2$, so that
it might be necessary to take a very large $L$.

The union of $C'$, $H$ and $H'$ 
has the required properties and has this $(\ell,L)$ butterfly curve as its product boundary.
\end{proof}

Note that for each $\ell<\pi$, one could in principle compute
an explicit optimal bound $L\ge f(\ell)$, because the 
hyperbolic diameter $d$ of $D$ can be computed in terms of $\ell$,
and the distance $r$ from $H(c,a,b)$ to its limiting flat
can be computed in terms of $L$. Choosing $q_1,q_2,q_3,q_4$
adequately, we only need $2r<D$; but these computation would probably
not yield an explicit bound, as both $d$ and $r$ are obtained
from $\ell$ and $L$ through integrals without closed form.

\subsection{Horizontal annuli} 
The horizontal catenoids discussed earlier are minimal, but not minimizing, surfaces which
have boundary $\sigma = \del \HH^2 \times \{t_1, t_2\}$, where $|t_2 - t_1| < \pi$.  Conversely,
any minimal surface $\Sigma$ with $\del \Sigma$ equal to a pair of horizontal circles
must be one of these horizontal catenoids, or a translate of one by a horizontal isometry
of $\HR$, i.e., an isometry which fixes the $t$ component, see \cite{Ferrer-Martin-Mazzeo-Rodriguez}. 

This suggests that there should be a reasonable existence theory for minimal surfaces which
have boundary equal to a union of disjoint two curves $\sigma_1$ and $\sigma_2$, which
are each vertical graphs on $\del \HH^2$, and which are not too far apart. 
Recent work by the second author, Ferrer, Martin and Rodriguez \cite{Ferrer-Martin-Mazzeo-Rodriguez} shows that
this is indeed the case.  We say that the two curves $\sigma_1$ and $\sigma_2$ are
a short pair if the separation between them along any vertical line on $\del \HH^2$
is less than $\pi$. Then \cite{Ferrer-Martin-Mazzeo-Rodriguez} shows that there is an open dense set in the space
of short pairs which bound a minimal annulus.  However, not every short pair is
minimally fillable. One example, explained in \cite{Ferrer-Martin-Mazzeo-Rodriguez}, is if the $\sigma_j$
are the intersections of $S^1 \times \RR$ with a pair of planes, one sloping up and
the other down, separated by a vertical distance less than $\pi$. 

\subsection{Vertical surfaces of finite total curvature} 
Another class of minimal surfaces in $\HR$ are those which have vertical ends, i.e., ends which
are asymptotic to unions of vertical flats. The prototype and building blocks for these surfaces 
are the vertical catenoids.  The following result is proved in \cite{Martin-Mazzeo-Rodriguez}.
\begin{prop}
For each $g \geq 0$ there is a $k_0 = k_0(g)$ such that if $k \geq k_0$, then there exists a properly 
embedded minimal surface with finite total curvature and vertical ends in $\HH^2 \times \RR$, with
genus $g$ and $k$ ends.
\end{prop}
These surfaces are obtained by gluing together vertical catenoids along vertical strips which 
are positioned far from the neck regions of each catenoid.  

It is not yet clear what the sharp constant $k_0(g)$ should equal.  Using different gluing
constructions, it now seems possible that there might exist vertical minimal surfaces 
with any genus $g$ and with only $3$ ends, but this remains conjectural at present.

\section{Minimally fillable sets on the geodesic boundary}

We now study the asymptotic behavior of properly embedded minimal surfaces in $\HR$ at
the geodesic boundary, i.e., we study $\Gbound \Sigma=\Gcomp{\Sigma}\cap \Gbound{\HR}$, 
where $\Gcomp{\Sigma}$ is the closure of $\Sigma$ in $\Gcomp{\HR}$. The main result 
is that $\Gbound \Sigma$ can be quite wild, and in particular can be quite far from being 
a curve. We exhibit a significant restriction as to which closed subsets of $\Gbound{\HR}$
can appear as $\Gbound \Sigma$ for some $\Sigma$; we also obtain a complete
classification of the subsets of this type which are embedded curves. We also give examples,
some of which show that the extreme types of behavior allowed by Theorem \ref{theo:oscillations} 
actually do occur.

\subsection{Oscillations of linear minimal surfaces}

\begin{theo}\label{theo:oscillations}
Let $\Sigma$ be a complete, properly embedded minimal surface in $\HR$.  Then for any
$q\in\partial \Hyp^2$, the intersection of $\Gbound \Sigma$ with the Weyl chamber
$W^\pm(q)$ is either empty, equal only to the pole $p^\pm$ at the end of that Weyl chamber, 
or else is a closed interval containing $q$.
\end{theo}

What this result means is that if there exists a diverging sequence of points $(q_n, t_n) \in \Sigma$ 
with $\lim_n t_n/d_\Hyp(q_n,q_0)= r$, then for any $r' \in [0,r]$, there must exist another diverging 
sequence of points $(q_n', t_n') \in \Sigma$ with $\lim_n t_n'/d_\Hyp(q_n',q_0') = r'$.    This may
happen above all equatorial points $q$ lying on some arc $c \subset \del \HH^2$, which means
that $\Sigma$ must then oscillate quite wildly in the corresponding sector of $\HR$.
This behavior may seem surprising, but is known to happen for orbits of certain discrete 
groups of isometries.  We use this observation to construct minimal surfaces exhibiting similar 
oscillations.

A somewhat similar oscillation phenomenon has already been noticed by Wolf \cite{Wolf}.
He produces examples of minimal surfaces which attain a given height above a ray in
$\HH^2$ infinitely many times. The oscillations here are slightly different since 
the heights are growing linearly rather than being just bounded. However, it would not be surprising
if the two phenomena were related. 

\begin{proof}
Assume that $\Gbound\Sigma$ meets the interior of $W^+(q)$. 

We first show that $\Gbound \Sigma$ must contain $q$. If this were not the case, then
there would exist a small equatorial arc $c$ containing $q$ and a neighborhood $U$ of $c$
in $\Gcomp{\HR}$ which does not intersect $\Sigma$. 

Consider a hyperbolic barrier $H=H(c,0,\ell)$ with $\ell$ sufficiently close to $\pi$
to ensure that $H$ stays far from the flat $F$ defined by $c$, and hence ensures 
that $H\subset U$, hence $H \cap \Sigma = \emptyset$. 

Now, consider the vertical translates $H_t=H(c,t,t+\ell)$ of $H$. Any first contact between $H_t$ 
and $\Sigma$ would have to happen in the interior, since the geodesic boundaries of these surfaces 
are independent  of $t$ and hence `uniformly disjoint'. By the maximum principle, this interior
contact is impossible, so $\Sigma$ does not meet any of the $H_t$. The union $\cup_t H_t$ equals
the set of all points of distance at least $\rho_0$ from $F$ on the side of $c$. This is incompatible 
with the fact that $\Gbound\Sigma$ meets the interior of $W^+(q)$.

We can adapt this same argument to prove that in fact, $\Gbound\Sigma$ must meet $W^+(q)$ along 
an interval containing $q$ (see Figure \ref{fig:linear}). Indeed, otherwise there would be exist some 
$(q,r)\in W^+(q) \setminus \Gbound\Sigma$ and some $r'>r$ with $(q,r')\in\Gbound\Sigma$. As before,
there would then be a neighborhood $\mathcal V$ of  $(q,r)$ in $\Gcomp{\HR}$ which does not 
intersect $\Sigma$. Fix an arc $c\subset\partial\Hyp^2$ containing $q$ such that $c\times \{r\}\subset 
\mathcal V$. Define, for any fixed $p_0 \in \HH^2$, 
\[
D=\{(p,t) \,|\, t \geqslant r d_\Hyp(p,p_0)\}, \quad \mbox{so that}\quad \del D = \{(p,t): tz=r d_\Hyp(p,p_0)\}.
\]

\begin{figure}[htb]
\labellist
\small \hair 3pt
\pinlabel $\Sigma'$ [br] at 80 210
\pinlabel $\Sigma$ [t] at 76 85
\pinlabel $H_t$ [l] at 206 165
\pinlabel $q$ [l] at 177 129
\pinlabel $(q,r)$ [bl] at 150 198
\pinlabel $(q,r')$ [bl] at 134 215
\pinlabel $V$ [r] at 109 181
\pinlabel {$\partial D$} [tr] at 81 133
\endlabellist
\centering
\includegraphics[scale=.8]{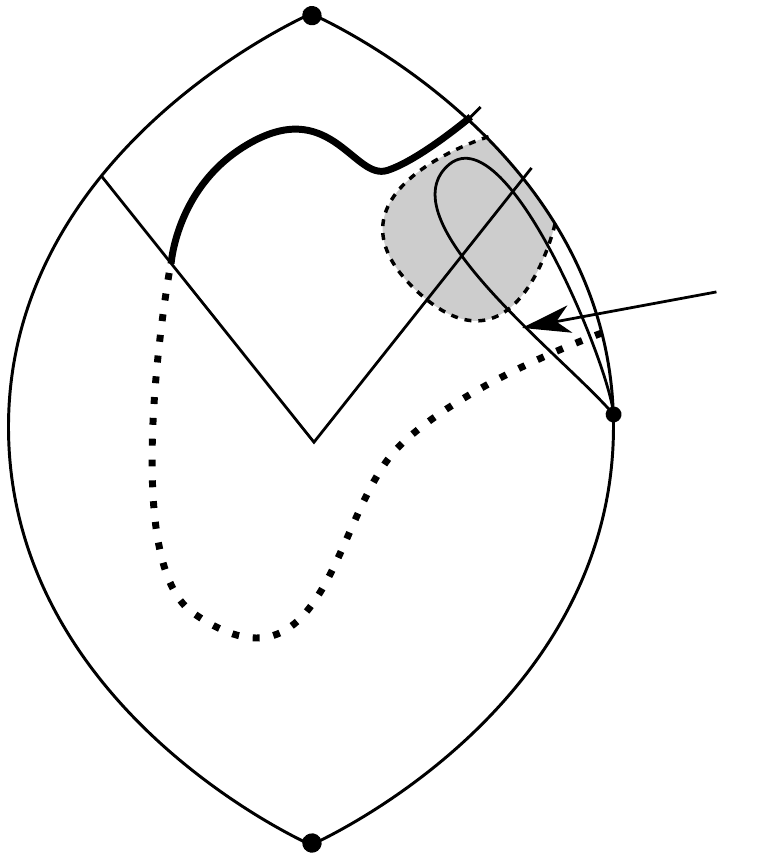}
\caption{Touching the upper part of $\Sigma$ with a hyperbolic barrier.}
\label{fig:linear}
\end{figure}
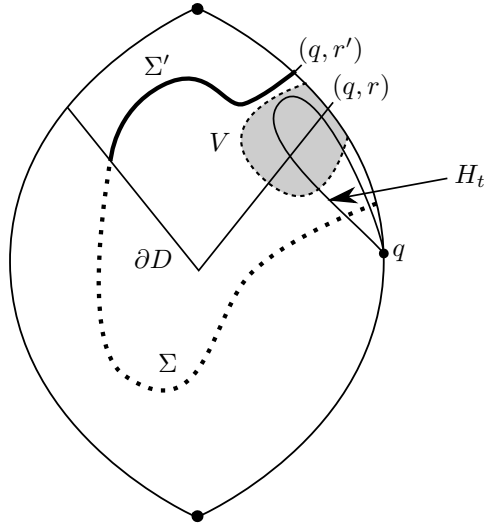

Replace $\Sigma$ by $\Sigma'=\Sigma\cap D$; this has an interior boundary contained in $\del D$ 
which lies outside $\mathcal V$. In particular $\Gbound\Sigma'$ intersects $\Gbound\HR$ only `above'
slope $r$, while the interior boundary of $\Sigma'$ is disjoint from the half space determined by some
subarc $c' \subset c$, $c' \ni q$.  Considering the same translating family of hyperbolic barriers 
$H_t=H(c',t,t+\ell)$ as before, now with $\ell$ small enough to ensure that $H_0\cap\Sigma=\varnothing$,
we see that there must  be an interior first point of contact between $\Sigma'$ and one of the $H_t$. This is 
a contradiction.
\end{proof}

As a straightforward consequence, we classify the curves $\sigma\subset \Gbound\HR$ which are minimally fillable. As before, by curve we
implicitly mean a finite union of disjoint 
\emph{Jordan curve}, i.e., of topological embedding of the circle.
\begin{coro}
Any minimally fillable curve in $\Gbound\HR$ is one of the following:
\begin{enumerate}
\item\label{enum:gfill1} the equator $\del\HH^2$;
\item\label{enum:gfill2}  the union of an equatorial arc $c=\overset{\frown}{q_1q_2}$ and 
two Weyl chambers $W^\epsilon(q_1)$, $W^\epsilon(q_2)$, $\epsilon\in\{+,-\}$;
\item\label{enum:gfill3}  the union of four Weyl chambers $W^+(q_1)$, $W^+(q_2)$,
$W^-(q_3)$ and $W^-(q_4)$ and two disjoint equatorial arcs $\overset{\frown}{q_2q_3}$ $\overset{\frown}{q_4q_1}$,
which possibly reduce to single points;
\item\label{enum:gfill4} the union of two curves of type \ref{enum:gfill2} with different $\epsilon$ and disjoint $c$.
\end{enumerate}
\end{coro}
In this result, the notation $\overset{\frown}{q_1q_2}$ denotes either of the two geodesic arcs joining $q_1$ and $q_2$,
and no orientation is specified. Note also that there is no requirement for the minimal surface filling a curve to be connected.

\begin{proof}
We deduce first from Theorem \ref{theo:oscillations} that any connected, minimally fillable curve $\sigma\subset\Gbound\HR$ 
is of type \ref{enum:gfill1} \ref{enum:gfill2} or \ref{enum:gfill3}. If $\sigma$ does not meet any Weyl chamber except along the
equator, then it must be contained in the equator. Since $\sigma$ is assumed to be a Jordan curve, it must be equal
to the equator, hence we are in case \ref{enum:gfill1}. 

Otherwise, $\sigma$ contains an interior point of a Weyl chamber, $(q_1,r)\in W^\epsilon(q_1)$ for some $q_1\in\del\HH^2$,
$\epsilon\in\{+,-\}$ and $r\in (0,1)$. By Theorem \ref{theo:oscillations}, $\sigma$ thus contains the segment 
$[0,r]$ in that Weyl chamber. Since $\sigma$ is a curve, it cannot terminate or turn (for in that case, 
it would contain a union of segments in Weyl chambers, and hence have interior), so it contains the entire Weyl chamber 
$W^\epsilon(q_1)$.  Once it reaches $p^\epsilon$, it must follow another Weyl chamber $W^\epsilon(q_2)$. If these 
are the only Weyl chambers which meet $\sigma$, then we are in case \ref{enum:gfill2}; namely, to close up,
$\sigma$ must follow an equatorial arc to return from $q_2$ to $q_1$. 

Suppose then that $\sigma$ meets (and hence contains) more than two Weyl chambers. Suppose we start at
$q_1$ and travel up to $p^+$ with $W^+(q_1)$, then traverse $W^+(q_2)$ down to the equator. After possibly
following an equatorial arc to $q_3$, $\sigma$ must contain the chamber $W^-(q_3)$; it cannot contain
$W^+(q_3)$ since in that case $\sigma$ would self-intersect at $p^+$. It then contains some other
$W^-(q_4)$, and once again to avoid self-intersections cannot contain any other Weyl chamber, hence
the only alternative is that it contains the final equatorial arc connecting $q_4$ to $q_1$. Thus 
we are in case \ref{enum:gfill3}. 

If $\sigma$ is not connected, then each one of its connected components is one of the types 
\ref{enum:gfill1}, \ref{enum:gfill2} or \ref{enum:gfill3}.  If one component were the equator 
then the other component would have to intersect it. Similarly, if one component is of type
\ref{enum:gfill3}, then any other component would have to intersect it at one of the poles. 
The only possibility is that $\sigma$ is a union of two disjoint curves of type \ref{enum:gfill2}.
It is clearly impossible for $\sigma$ to have more than two components. 

It remains to show that each of these cases is actually realized as $\Gbound\Sigma$.

The geodesic boundary of any horizontal slice $\HH^2\times \{t\}$ achieves case \ref{enum:gfill1}. 

Next, case \ref{enum:gfill2} is realized by the semi-infinite tall rectangles $H(c,a,+\infty)$, or $H(c,-\infty,b)$.
The union of two tall rectangles with $b<a$ is a disconnected embedded minimal surface which fills
the disjoint union of two curves of this type. 

Finally, noting the relationship between the product and geodesic boundaries, the minimal surfaces described in 
Proposition \ref{theo:fillability} realize case \ref{enum:gfill3}. 
\end{proof}

\subsection{Examples}
We recall a number of examples, all previously known, which illustrate 
Theorem \ref{theo:oscillations} in different ways.

\begin{exem}
First recall the Scherk-type minimal surfaces constructed first in \cite[Section 4]{Nelli-Rosenberg},
and latter by \cite{Mazet-Rodriguez-Rosenberg}. 
These are obtained by solving a Dirichlet problem on a regular polygon with angles $\pi/2$,
imposing either finite or infinite boundary conditions on each side, and symmetrizing (by reflection
across the boundary) to obtain complete embedded surfaces.  For the surfaces constructed in this way,
$\Gbound \Sigma = \Gbound \HR$. For example, the 
surface described in of \cite[Remark 3]{Nelli-Rosenberg} contains vertical lines above every point of
a lattice $\Lambda\subset \HH^2$, the limit set of which is the entire boundary $\del\HH^2$, from
which the assertion about $\Gbound \Sigma$ follows. 

Note that in these examples, the product boundary is also quite degenerate: 
$\Pbound\Sigma$ contains $\del\HH^2\times\mathbb{R}$.
\end{exem}

\begin{exem}
Next consider the helicoid $\Sigma$ in $\HR$. This surface is invariant by a one-parameter 
group of isometries which acts by rotations on $\HH^2$ and simultaneiously by translations in
$\mathbb{R}$.  It is not hard to see that $\Gbound\Sigma=\Gbound \HR$, but on the other hand,
$\Pbound\Sigma$ is the union of a smooth helix in $\del\HH^2\times\mathbb{R}$ and the 
two caps. 
\end{exem}

Other examples with less extreme behavior have geodesic boundary which is the 
union of an equator and of a number of arcs in Weyl chambers.

\begin{exem}
Fix a one-parameter group of isometries $\varphi_t$ of $\HR$, such that $\varphi_t$ acts by
translation along a geodesic $\gamma$ on $\HH^2$, and by (nontrivial) translation on $\mathbb{R}$.
This is a `diagonal' translation of $\HR$. Let $\sigma$ be a geodesic in $\HH^2$ which is orthogonal to 
$\gamma$, and let $\Sigma = \sqcup_{t} \varphi_t(\sigma \times \{0\})$. 

The fact that $\Sigma$ is minimal can  be argued as follows. It is invariant with respect to the 
geodesic symmetry around each translate of $\sigma$, hence its mean curvature vector is invariant
under these symmetries, and must therefore vanish. One checks easily that $\Gbound \Sigma$ is 
the union of the equator and of two arcs in the Weyl chambers lying over the endpoints of
$\gamma$; the lengths of these arcs are equal and are determined by the ratio of the translation lengths 
of the vertical and horizontal parts of $\varphi_t$.
\end{exem}

\begin{exem}
By the method of \cite{Nelli-Rosenberg}, one can construct a minimal graph over $\HH^2$ which
has boundary values given by the graph of an unbounded function $\sigma:\del\HH^2\to \mathbb{R}\cup\{+\infty\}$
which is finite and smooth on $\del \HH^2 \setminus \{q\}$, and with $\sigma(q) = +\infty$.
Adjusting the rate of divergence of $\sigma$ near $q$, one can obtain a properly embedded minimal disk $\Sigma$
such that $\Gbound\Sigma$ is the union of the equator and of an arc of any length in the Weyl chamber $W^+(q)$. 
It is also possible to find such surfaces where this arc is empty or equal to the entire Weyl chamber.
\end{exem}

We now give a class of examples showing that $\Gbound\Sigma$ can be as wild as in the conclusion of Theorem \ref{theo:oscillations}
without being equal to all of $\Gbound \HR$.
\begin{theo}\label{theo:examples}
For any $r_0> 0$, there exists a complete minimal surface $\Sigma\subset\HR$ such that $\Gbound \Sigma$ is the 
entire strip $\{(q, r): |r|\leqslant r_0, q \in \del \HH^2\}$.
\end{theo}

\begin{proof}
We shall construct $\Sigma$ to be periodic with respect to some discrete subgroup $\Gamma$
of isometries of $\HR$. Indeed, let $\Gamma = \pi_1(S_g)$, where $S_g$ is a compact, orientable 
genus $g$ surface, and choose any discrete and faithful representation $\alpha:\Gamma
\to\mathrm{PSL}(2;\mathbb{R})$.

There is a standard presentation of $\Gamma$ as the quotient of the free group $\calF$ on generators 
$a_1,\dots,a_{2g}$ by a normal subgroup $R$ generated by some product of commutators 
$[a_i,a_j]=a_i^{-1}a_j^{-1}a_ia_j$. Any homomorphism $\calF \to \RR$ sends all commutators to $0$,
and hence descends to a homomorphism $\Gamma \to \RR$. We apply this to the homomorphism
which sends $a_1$ to $1$ and all other $a_i$, $i > 1$, to $0$, and denote the induced homomorphism
on $\Gamma$ by $\beta$, so $\beta(a_1)=1$ and all other $\beta (a_i) = 0$. We identify $a_i$ with 
its image under the quotient map.

Since $\alpha(a_1)$ is a hyperbolic element of $\mathrm{PSL}(2; \RR)$, it fixes a unique geodesic
$\gamma_1$ in $\HH^2$, and acts by a translation of $\tau$ along this geodesic. Consider now 
the action of $\Gamma$ on $\HR$ defined by 
\[
A: \Gamma \to \mathrm{Isom}\, (\HR), \quad  A(\gamma) = (\alpha (\gamma),\frac{r_0}{\tau}\cdot\beta(\gamma)). 
\]
The quotient $\HR/\Gamma$ associated to this action is diffeomorphic to the product $S_g\times \RR$. 
Now choose another representation $B: \ZZ \to \mathrm{Isom}\,(\HR)$ which commutes with $A$, 
and is such that the associated quotient $\HR/ (\Gamma \times \ZZ)$ is diffeomorphic to $S_g \times S^1$. 
For example, we can simply let $B(k)(p,t) = (p, t+1)$. With respect to this identification, consider the incompressible embedding
$f: S_g \hookrightarrow S_g \times \{0\} \subset S_g \times S^1$. According to a theorem of Schoen and Yau \cite{Schoen-Yau},
there is an immersion $\hat{f}: S_g \to S_g \times S^1$ which is isotopic to $f$ and such that the image
of $\hat{f}$ has least area in this isotopy class. Let us denote this surface by $\hat{S}$. By a further result of 
Freedman-Hass-Scott \cite{FHS}, $\hat{S}$ is actually embedded. It is then clear that we may pass to the 
$\ZZ$ cover and consider $\hat{S}$ as a compact embedded minimal surface in $\HR/\Gamma$. 
Finally, let $\Sigma$ be the inverse image of $\hat{S}$ in the universal cover $\HR$. Thus $\Sigma$ is 
a minimal surface which is invariant under the action of $\Gamma$.

Now consider the behavior of $\Sigma$ above the geodesic $\gamma_1$, or
in other words, the behaviour of the intersection $(\gamma_1 \times \RR) \cap \Sigma$.
Denote the endpoints of $\gamma_1$ by $q_1,q_2$.  If $(z_0,t_0) \in \Sigma$ with $z_0 \in \gamma_1$,
then the entire sequence $(z_n,t_n):=A(a_1^n)\cdot(z_0,t_0)$ lies in $\Sigma$.  Observe that
$t_n=(r_0/\tau) n$ and $d_\Hyp(z_0,z_n)=n \tau$, $n \in \mathbb{Z}$. 
It follows that $\Gbound\Sigma$ contains both $(q_1,-r_0)$ and $(q_2,r_0)$.

Suppose finally that $q$ is any point on $\partial \Hyp^2$. Since $\Gamma$ acts cocompactly
on $\Hyp^2$, $q$ lies in its limit set so we can find a sequence $h_n\in\Gamma$ such 
that the endpoints $q_1^n,q_2^n$ of the geodesic $\alpha(h_n)\cdot\gamma_1$ both limit to $q$.
The asymptotic behavior of $\Sigma$ over the geodesic $\alpha(h_n)\cdot\gamma_1$ is conjugate 
to its behavior over $\gamma$, hence $\Gbound\Sigma$ contains $(q_1^n,-r_0)$ and $(q_2^n,r_0)$
for all $n$. Since $\Gbound\Sigma$ is closed, it must contain $(q,-r_0)$ and $(q,r_0)$; using 
the previous Theorem, it must contain the whole interval $q\times[-r_0,r_0]$.

We may of course assume that $\tau$ is the systole of $\alpha$, i.e., the smallest translation length 
amongst all elements of $\alpha(\Gamma)$. Then it is clear from the construction that $\Gbound\Sigma$
is contained in the strip $\{|r|\leqslant r_0\}$, which completes the proof.
\end{proof}

Theorems \ref{theo:oscillations} and  \ref{theo:examples} were inspired by Benoist's article \cite{Benoist},
where a similar behavior is exhibited in some linear groups. It is natural to ask if any similar statements 
hold for minimal hypersurfaces (or minimal submanifolds of codimension greater than one) in more general 
symmetric spaces.   Since $\HR$ has a Euclidean factor, an equatorial point may be interpreted either as
a boundary point for a Weyl chamber as we do here, or as the center of a larger chamber.  It is not clear which 
of these would be part of the most natural to conjecture in the higher rank setting.  

\section{Open questions: minimal submanifolds in nonpositively curved symmetric spaces}
We conclude with some remarks and open questions which are meant to convey the range of unexplored territory.

\subsection{Minimal surfaces of $\HxR$}
The results quoted and proved above leave open some questions about fillability of boundary curves in 
$\del \HxR$.
\begin{prob}
Characterize the curves of $\Pbound \HxR$ which are minimally fillable.
\end{prob}
In particular, what can one say about curves which are connected and tall in the vertical part of the boundary, 
and for which the intersection with each cap is a pair of disjoint geodesics? Some of these can be filled
by minimal surfaces of Jenkins-Serrin type, cf.\ \cite{Nelli-Rosenberg}, but that method does not give
as much flexibility in the choice of vertical components as in Proposition \ref{theo:fillability}. 

\begin{prob}
Characterize the closed subsets of $\Gbound \HxR$ which arise as $\Gbound \Sigma$ for an embedded minimal surface
$\Sigma\subset \HxR$.
\end{prob}
Theorem \ref{theo:oscillations} provides a strong restriction, but is far from definitive.

\subsection{Higher-dimensional symmetric spaces}
As noted in the introduction, the starting point for our work was to consider $\HxR$ as a 
particular type of nonpositively curved symmetric spaces.  This perspective leads naturally 
to the question of understanding the properly embedded minimal submanifolds and the
asymptotic Plateau problem for other such symmetric spaces. 

The next simplest case is the product of two hyperbolic planes, possibly with 
different curvatures, $\HH^2_\kappa\times \HH^2_\lambda$.  It is likely that the 
minimal surfaces in these might have rather different behavior depending on 
whether the curvatures $\kappa$, $\lambda$ are equal to one another or not.

The space $\HH^2_\kappa\times \HH^2_\lambda$ has both a product and a geodesic boundary.
The paper \cite{Mazzeo-Vasy} explores the role of these two boundaries in linear elliptic theory.
The natural question in the present setting is then
\begin{prob}
Determine which curves or which closed surfaces in 
\[
\Pbound(\HH^2_\kappa\times \HH^2_\lambda) \quad\mbox{and} \quad
\Gbound(\HH^2_\kappa\times \HH^2_\lambda)
\]
are minimally fillable by surfaces or $3$-dimensional submanifolds.
\end{prob}

We also pose the 
\begin{prob}
Determine which curves of $\del \HH^2_\kappa \times \del \HH^2_\lambda$ are
the boundary of an embedded minimal surface of $\HH^2_\kappa\times \HH^2_\lambda$.
\end{prob}
While the complete characterization may be very difficult, one might construct examples 
using pairs of representations of a surface group $\Gamma$ into $\mathrm{PSL(2,\mathbb{R})}$
to construct periodic examples, similar to what we did in the proof of Theorem 
\ref{theo:examples}. 

Similar questions can be posed in other symmetric spaces, e.g.\ $\mathrm{SL}(3)/\mathrm{SO}(3)$.
A number of different compactifications have been studied, see \cite{Borel-Ji}, \cite{Mazzeo-Vasy2}.
\begin{prob}
Determine a compactification of a nonpositively curved symmetric space 
which is particularly well suited
for the study of the asymptotic Plateau problem. 
\end{prob}
This is of course only loosely stated, since it is not clear what `well suited' should mean.
Presumably this entails that many submanifolds of that compactification are minimally fillable. 

Note that the rank one case is basically covered by a more general
result of Lang \cite{Lang}.

\bibliographystyle{smfalpha}
\bibliography{biblio.bib}


\end{document}